\documentclass[]{amsart}
\usepackage{amsfonts,amssymb,amsmath, amsthm}
\usepackage{color}
\theoremstyle{plain}

\newtheorem{theorem}{Theorem}[section]
\newtheorem{definition}[theorem]{Definition}

\newtheorem{lemma}[theorem]{Lemma}
\newtheorem{proposition}[theorem]{Proposition}

\newtheorem{preremark}[theorem]{Remark}
\newenvironment{remark}{\begin{preremark}\normalfont}{\end{preremark}}

\allowdisplaybreaks

\definecolor{myblue}{cmyk}{1,0.7,0,0}
\definecolor{mylightblue}{cmyk}{0.5,0.3,0,0}
\definecolor{mypaleblue}{cmyk}{0.08,0.05,0,0}
\definecolor{mypurple}{cmyk}{0.6,1,0,0}
\definecolor{myorange}{cmyk}{0,0.83,1,0.5}
\definecolor{mygreen}{cmyk}{0.9,0,1,0.2}
\definecolor{myred}{cmyk}{0,1,1,0.1}
\definecolor{mycolor}{cmyk}{0,0.5,1,0.5}
\definecolor{mymint}{cmyk}{0.93,0,0.75,0}

\newcommand{\A}{\mathcal A}
\newcommand{\B}{\mathcal B}

\renewcommand{\l}{\mathcal L}
\renewcommand{\r}{\mathcal R}

\newcommand{\e}{\mathcal E}
\newcommand{\F}{\mathcal F}

\renewcommand{\H}{\mathcal H}

\newcommand{\R}{\mathbb R}

\title[Quasilinear Parabolic Equations]{Local Behavior of Solutions of Quasilinear Parabolic Equations on metric spaces}
\author{Janna Lierl}

\begin{document}

\begin{abstract}
We introduce a notion of quasilinear parabolic equations over metric measure spaces. Under sharp structural conditions, we prove that local weak solutions are locally bounded and satisfy the parabolic Harnack inequality. Applications include the parabolic maximum principle and pointwise estimates for weak solutions. 
\end{abstract}

\maketitle

\section{Introduction}
In their 1967 paper \cite{AS67}, Aronson and Serrin proved the parabolic Harnack inequality for weak solutions $u = u(t,x)$ of the quasilinear equation
\begin{align*}
\int_Q\int \left( -u \phi_t + \phi_x \cdot \A(t,x,u,u_x) \right) dx \, dt = \int_Q\int \phi \, \B(t,x,u,u_x) dx \, dt,
\end{align*}
provided that $\A$ and $\B$ satisfy certain structural conditions. Here, $x$ is in Euclidean space $\R^n$, $u_x$ is the spatial gradient of $u$, and $\phi= \phi(t,x)$ is any continuously differentiable test function having compact support in $Q$. 

At about the same time, Trudinger \cite{Tru68} and Ladyzhenskaja, Solonnikov and Uralceva \cite{LSU68} proved very similar results. 

The present paper introduces a notion of quasilinear equations over metric measure spaces and proves the parabolic Harnack inequality under certain hypotheses on the structure of the equation and natural conditions on the geometry of the underlying space. In particular, the parabolic Harnack inequality holds for quasilinear equations on metric measure spaces that satisfy volume doubling, Poincar\'e inequality and the cutoff Sobolev inequality. In the case of the linear heat equation, these are known to be equivalent to the parabolic Harnack inequality, as well as to sharp two-sided heat kernel estimates  \cite{Gri91,SC92,SturmIII,BBK06}. %More generally, it suffices to assume a weighted Sobolev inequality and a weighted Poincar\'e inequality. 

Concerning the structural hypotheses, we follow \cite{AS67, LSU68, Tru68} by assuming that a quasilinear equation should be represented in terms of some ''divergence form part" $\A$ and a ``lower order part" $\B$. This is somewhat contrary to the approach in \cite{SturmII, LierlPHI, LierlPHIf} that was based on the bilinearity and a structural decomposition of bilinear forms, in addition to quantitative inequalities.

Though the main interest of this work is likely to be in the context of a reference Dirichlet space, some of our results - the mean value estimates and the local boundedness of weak solutions - apply also to subelliptic operators such as the Kolmogorov-Fokker-Planck operator. 

The hypotheses that we impose on the structure of the equation are, in a certain sense, sharp. This is because we use Lorentz spaces rather than $L^p$-spaces.
For the special case of quasilinear operators on Euclidean space this means that we recover the parabolic Harnack inequality of Aronson - Serrin but under slightly weaker  - and sharp - integrability conditions on the coefficients. 

There already exists a quite broad literature that applies the parabolic Moser  iteration \cite{Mos64} in a non-Euclidean, linear setting, beginning with \cite{Gri91,SC92,SC02}  on Riemannian manifolds, \cite{SturmII, DeLeva, LierlPHI} on Dirichlet spaces that admit a carr\'e du champ, \cite{LierlPHIf} on fractal-type Dirichlet spaces. The proof in the present paper is based on some of these earlier works (as well as \cite{AS67}) but self-contained and aims to give full attention to technical issues pertaining to the existence and local boundedness of weak solutions, the appropriate function spaces, and the structural hypotheses.  

A different direction concerns the generalization of the equation rather than the underlying space. A class of degenerate elliptic operators (so-called generalized Kimura diffusion operators) is studied in \cite{EpMa16} and covered by the setting of the present paper. For quasilinear subelliptic operators, the parabolic Harnack inequality is proved in \cite{CaCiRe13}.

Degenerate subelliptic operators such as those of Kolmogorov-Fokker-Planck type are of interest as they indicate the margin of the wide scope of Moser's iteration. While mean value estimates follow by Moser iteration (see, e.g., \cite{CiPo08}),  the second part of Moser's iteration does not seem to apply due to the lack of a Poincar\'e inequality as well as a lack of the proper structure of the equation (cf.~ our hypothesis H.2 and Section \ref{ssec:Kolmogorov}).

We mention that there are alternative ways to obtain the parabolic Harnack inequality from volume doubling and Poincar\'e inequality, for instance by using elliptic Moser iteration, see e.g. \cite{BBK06}. This is of interest especially in time-independent settings.
A variational approach to the parabolic Harnack inequality on metric measure spaces is taken in \cite{KMMP12} under the hypothesis that weak solutions (i.e.~parabolic minimizers) already satisfy the Cacciopoli-type estimates.

Our main results are in part motivated by an application to the study of heat kernels on inner uniform domains similar to \cite{GyryaSC}. For certain non-symmetric heat kernels, Doob's transform yields a heat equation whose structure is not covered by \cite{LierlPHI} due to unbounded coefficients, but does satisfy our structural hypotheses H.1 and H.2.

\textbf{Structure of the paper.} In the first part of the paper we introduce the notion of quasilinear equations on abstract spaces (Section \ref{sec:quasilinear}), prove Cacciopoli-type estimates and mean value estimates (Section \ref{sec:MVE}) and the parabolic Harnack inequality (Section \ref{sec:PHI}).
 
In the second part of the paper (Section \ref{sec:examples}) we discuss examples. First, we apply our main results to Dirichlet spaces satisfying volume doubling and Poincar\'e inequality. 
We show that if a quasilinear form is ``adapted" to a reference Dirichlet form (see Definition \ref{def:adapted form}), then our structural hypotheses H.1, H.2 are satisfied. Being adapted to a Dirichlet form is a property that should be easy to check in applications. In the Dirichlet space context, we provide several applications of the Harnack inequality: the H\"older continuity of weak solutions, the parabolic maximum principle, and pointwise estimates for weak solutions.

Second, we apply our results to quasilinear operators on Euclidean space and discuss the sharp conditions on the coefficients in comparison to the structural hypotheses in \cite{AS67}.

Third, we consider a Kolmogorov-Fokker-Planck operator. This example illustrates that there is an actual difference between our hypothesis H.1 and hypothesis H.2.

Finally, we emphasize the relevance of the metric measure space setting by combining our main result with a metric measure transform.

\textbf{Acknowledgement.}
TBA

\section{Quasilinear forms, structural and geometric hypotheses} \label{sec:quasilinear}

\subsection{Quasilinear forms and weak solutions}
Let $X$ be a locally compact separable Hausdorff space and $\mu$ a locally finite Borel measure with full support on $X$. 

Let $\F$ be a linear subspace of $L^2(X,\mu)$ such that
\begin{enumerate}
\item
$\F$ is dense in $L^2(X,\mu)$.
\item 
There is a norm $\| \cdot \|_{\F}$ so that $(\F,\|\cdot\|_{\F})$ is a Banach space and $\| f \|_{\F} \ge \| f \|_{L^2}$ for every $f\in \F$.
\item
$\F \cap \mathcal C_{\mbox{\tiny{c}}}(X)$ is dense in $(\mathcal C_{\mbox{\tiny{c}}}(X),\| \cdot \|_{\infty})$ and dense in $(\F,\|\cdot\|_{\F})$.
\item 
$\F_{\mbox{\tiny{b}}}:= \F \cap L^{\infty}(X)$ is an algebra.
\item 
If $f \in \F$ then $(f \vee 0) \in \F$ and $(f \wedge 1) \in \F$.
\item
If $f \in \F$ then $\Phi(f) \in \F$ for any function $\Phi \in \mathcal{C}^1(\R^m)$ with $\Phi(0)=0$, where $m$ is a positive integer.
\item 
$\F$ contains cutoff functions: for every open $U \subset X$ and every compact $K \subset U$ there exists a continuous function $\psi:X \to [0,1]$ in $\F$ that takes value $1$ on $K$ and value $0$ on $X \setminus U$.
\end{enumerate}
%  core dense in $\F$

Here, $\mathcal C$ denotes the space of continuous functions, $\mathcal C_{\mbox{\tiny{c}}}$ are continuous functions with compact support, and $\F_{\mbox{\tiny{c}}}$ will be the functions in $\F$ with compact support. Further, $a \wedge b = \min\{a,b\}$ and $a \vee b = \max\{a,b\}$.

\begin{definition} \label{def:quasilinear form}
We call a collection of maps $\e_t:\F \times \F \to \R$, $t \in \R$, a {\em quasi-linear form} if 
\begin{enumerate}
\item
there exist signed measure valued forms $\A_t$ and $\B_t$ such that
\[ \e_t(u,g) = \int d\A_t(u,g) + \int d\B_t(u,g), \quad \mbox{ for all } u,g \in \F. \]
\item the maps
$t \mapsto \int d\A_t(u,g)$ and $t \mapsto \int d\B_t(u,g)$ are measurable for all $u,g \in \F$.
\item (right-linearity)
$\A_t(u,g_1 + g_2) = \A_t(u,g_1) + \A_t(u,g_2)$ and $\A_t(u,cg_1) = c \A_t(u,g_1)$ for all $u,g_1, g_2 \in \F$, $c \in \R$. Similarly for $\B_t$.
\item ($\B_t$ is right-local)
$\B_t(u,g)= 0$ whenever $u,g \in \F$ with $g=0$ on the support of $u$.
\item ($\A_t$ is right-strongly local)
$\A_t(u,g)= 0$ whenever $u,g \in \F$ with $g$ constant on the support of $u$.
\item (right-product rule)
$d\A_t(u,fg)= g \, d\A_t(u,f) + f \, d\A_t(u,g)$ and
$d\B_t(u,fg)= g \, d\B_t(u,f)$ whenever $u,f,g, fg \in \F$.
\item
(right-chain rule)
for any $u, g_1, g_2, \ldots, g_m \in \F_{\mbox{\tiny{b}}}$, $g = (g_1, \ldots, g_m)$, and $\Phi \in \mathcal C^1(\R^m)$ with $\Phi(0)=0$, we have $\Phi(g) \in \F_{\mbox{\tiny{b}}}$ and
\begin{align} \label{eq:chain rule for Gamma}
d\A_t(u,\Phi(g)) = \sum_{i=1}^{m} \frac{\partial \Phi}{\partial x_i} (g) \, d\A_t(u,g_i).
\end{align}
%where $\Phi_{x_i}:=\partial \Phi / \partial x_i$ and $\tilde u$ is a quasi-continuous version of $u$, see \cite[(3.2.27) and Theorem 3.2.2]{FOT94}. 
\item (right-continuity)
For every $t \in \R$ there exists an open interval $I \ni t$ such that for any $u \in L^2( I \to \F) \cap \mathcal C(I \to L^2)$ and there is a constant $C(u,t)$ such that 
 \[ |\e_s(u(s,\cdot),g)| \le C(u,t) \| g \|_{\F} \quad \mbox{ for all } s \in I, g \in \F_{\mbox{\tiny{c}}}. \]
\end{enumerate}
\end{definition}

Let $I$ be a bounded open interval and $U\subset X$ open.
Let $L^2_{\mbox{\tiny{loc}}}(I \to \F;U)$ be the space of all functions $u:I \times U \to \R$ such that for any open interval $J$ relatively compact in $I$, and any open subset $A$ relatively compact in $U$, there exists a function $u^{\sharp} \in L^2(I \to \F)$ such that $u^{\sharp} = u$ a.e.~in $J \times A$. If all these $u^{\sharp}$ are in $\mathcal C(I \to L^2(U))$, then we write $u \in L^2_{\mbox{\tiny{loc}}}(I \to \F;U) \cap \mathcal C_{\mbox{\tiny{loc}}}(I \to L^2(U))$.

\begin{definition}
 Set $Q = I \times U$. A map $u: Q \to \R$ is a \emph{local weak subsolution} of the heat equation for $\e_t$ in $Q$ if
\begin{enumerate}
\item
$u \in L^2_{\mbox{\tiny{loc}}}(I \to \F;U) \cap \mathcal C_{\mbox{\tiny{loc}}}(I \to L^2(U))$,
\item 
For almost every $a,b \in I$ with $a < b$, and any non-negative $\phi \in \F_{\mbox{\tiny{c}}}(U)$,
\begin{align*}
 \int u(b)  \phi \, d\mu - \int u(a)  \phi \, d\mu + \int_a^b \e_t(u(t),\phi) dt
\leq 0.
\end{align*}
\end{enumerate} 
A map $u$ is a {\em local weak supersolution} if $-u$ is a local weak subsolution. If both $u$ and $-u$ are local weak subsolutions then $u$ is called a {\em local weak solution}.
\end{definition}
 
It is worth to remark that local weak {\em solutions} can equivalently be defined using weak time-derivatives, see \cite[Proposition 7.8]{LierlPHI}.

\subsection{Structural hypotheses} \label{ssec:structural hypotheses}
For the rest of the paper, we fix $\delta^* \in (0,1)$.
Let $(B_{\delta})$ be a collection of relatively compact open subsets of $X$ such that $B_{\delta'} \subsetneq B_{\delta}$ whenever $\delta^* \le \delta' < \delta \le 1$.

%Let $(B_{\delta})$ be relatively compact open subset of $X$ such that $B_{\delta'} \subsetneq B_{\delta}$ whenever $\delta^* \le \delta' < \delta \le 1$.

Fix $a < a' < b' < b$. 
For $\delta \in [\delta^*,1]$, let $I^-_{\delta} = (a-a_{\delta},b')$ be a strictly increasing sequence of bounded open intervals. Let $I^- = (a-a_1,b)$, $Q^- = I^- \times B_1$ and $Q^-_{\delta} = I^-_{\delta} \times B_{\delta}$.
%The reason for choosing $b' < b$ is that we want functions that are locally integrable over $I^-$ to be integrable over $I^-_{\delta}$ when $\delta<1$.
We also define $I^+_{\delta} = (a',b+a_{\delta})$, $I^+ = (a,b+a_1)$, $Q^+ = I^+ \times B_1$, and $Q^+_{\delta} = I^+_{\delta} \times B_{\delta}$.

We assume there are constants $C_3,k \in (0,\infty)$ such that
\begin{align} \label{eq:C_3}
|a_{\delta} - a_{\delta'}|^{-1} \le C_3 |\delta-\delta'|^{-k}
\end{align}
for all $\delta^* \le \delta' < \delta \le 1$. 

Let $\kappa \ge 0$ and define $\bar v = \max(v,0)+\kappa$ and $\bar v_{\varepsilon} = \bar v + \varepsilon$ for $\varepsilon \in (0,1)$.
For $p < 2$, let
\begin{align*}
\mathcal{H}(v) := 
\begin{cases}
 \frac{1}{p} \bar v^p, \quad & p \ne 0 \\
 \log \bar v, & p = 0.
 \end{cases}
\end{align*} 
Then $\mathcal  H$ is twice continuously differentiable on $(0,\infty)$.

For $p \ge 2$ and positive integers $n$, define also
\begin{align*}
\mathcal{H}_n(v) 
& =  \frac{1}{2} \bar v^2 (\bar v \wedge n)^{p-2} + \left( \frac{1}{p} - \frac{1}{2} \right) (\bar v \wedge n)^p - \bar v \kappa^{p-1}  + \frac{p-1}{p} \kappa^p.
\end{align*}
Then $\mathcal  H_n$ has one continuous derivative
\begin{align*}
\mathcal{H}_n'(v) = 
 \bar v (\bar v \wedge n)^{p-2} - \kappa^{p-1}
\end{align*} 
on $(0,\infty)$.
For non-negative functions $u$ we will write $u_n = u \wedge n$.

%Let $d\Gamma$ be an energy measure on $\F$. That is, 
%\begin{enumerate}
%\item
%$d\Gamma(\cdot,\cdot)$ is a measure-valued positive symmetric bilinear form with domain $\F$
%\item(Markovian) If $f \in \F$ then $d\Gamma(f^+ \wedge 1,f^+ \wedge 1) \le  d\Gamma(f,f)$.
%\item (strongly local)
%$d\Gamma(f,g) = 0$ whenever $g$ is constant on the support of $f$.
%\item (chain rule)
%$d\Gamma(\Phi(f),g) = \Phi'(f) d\Gamma(f,g)$ whenever $f,g \in \F$, $\Phi \in \mathcal C^1(\R^m)$ with $\Phi(0)=0$, $m$ positive integer. 
%\item
%$\int d\Gamma(u,u) \le C \|u \|_{\F}^2$.
%\end{enumerate}
%We {\em do not} assume that $\int d\Gamma$ is closed.

Fix $\eta \in (0,1)$. We say that H.1a, H.1b, or H.2, respectively, hold for $u \in L^2_{\mbox{\tiny{loc}}}(I \to \F;B_1)$, if there is a positive Radon measure $d\Gamma(u)$ and constants $\beta \in [1,\infty)$, $a \in (0,1)$ such that for all $\delta^* \le \delta' < \delta \le 1$ there exist constants $C_1=C_1(\delta',\delta) \in (0,\infty)$ and $C_2 \in (0,\infty)$, such that
\begin{align} \label{eq:subsol estimate decomposition}  \tag{\bf H.1a}
\begin{split}
& \quad \int_I \left[ -  \e_t( u, \mathcal H'_n(u) \psi^2 )  +  \frac{a}{2} \int \bar u_n^{p-2} \psi^2 d\Gamma(u)  
 +  \frac{p-2}{4}a \int_{\{\bar u \le n\}} \bar u_n^{p-2} \psi^2 d\Gamma(u) \right] \chi dt\\
& \le 
p C_1 |||\bar u \bar u_n^{\frac{p-2}{2}} \psi \chi^{1/2} |||_{I \times B_1}^2  
 + p^{\beta} C_2 |\delta'-\delta|^{-k}  \int_I \int \bar u^2 \bar u_n^{p-2} \psi \, d\mu \, \chi dt
\end{split}
\end{align}
for all positive integers $n$ and all $p \in [2,\infty)$, and any smooth function $\chi:I \to [0,1]$.
\begin{align}  \tag{\bf H.1b}
\begin{split}
& \quad \int_I \left[ \frac{1-p}{|1-p|} \e_t( u, \mathcal H'(u_{\varepsilon}) \psi^2 )  + \frac{|p-1|}{4}a \int \bar u_{\varepsilon}^{p-2} \psi^2 d\Gamma(u) \right] \chi dt \\
& \le 
(1 \vee |p|) C_1 ||| \bar u_{\varepsilon}^{\frac{p}{2}}\psi  |||_{I \times B_1}^2  
 + (1 \vee |p|^{\beta}) C_2 |\delta'-\delta|^{-k} \int_I \int \bar u_{\varepsilon}^p \psi d\mu \,  dt,
\end{split}
\end{align}
for all $\varepsilon \in (0,1)$, all $p \in (-\infty,0) \cup (0,1-\eta) \cup (1+\eta,2)$, any smooth function $\chi:I \to [0,1]$, provided that $u$ is non-negative and locally bounded.
\begin{align}  \tag{\bf H.2}
\begin{split}
& \quad - \e_t( u, \mathcal H'(u_{\varepsilon}) \psi^2 )  +  a \int \bar u_{\varepsilon}^{-2} \psi^2 d\Gamma(u) \\
& \le 
\sum_{i=1}^m D_i(t) \| \psi \|_{2r_i',2}^2 + C_2 |1-\delta^*|^{-k} \int \psi d\mu, \quad \mbox{ for a.e. } t \in I,
\end{split}
\end{align}
for all $\varepsilon \in (0,1)$, $p=0$, provided that $u$ is non-negative and locally bounded. Here, $m$ is a positive integer and  for each $i$, $D_i$ is in $L^{q_i}(I)$ and the pair $(r_i',q_i')$ has H\"older conjugates $(r_i,q_i)$ satisfying \eqref{eq:gamma}.

The function $\psi= \psi_{\delta',\delta}$ in H.1 and H.2 is a cutoff function for $B_{\delta'}$ in $B_{\delta}$ and we will assume throughout the paper that it is the same cutoff function as the one appearing in the weighted Sobolev inequality \eqref{eq:wSI}, see Definition \ref{def:wSI}. The norm $||| \cdot |||$ is defined below in \eqref{eq:||| norm}.

%We say that a constant depends on {\em the structure of the equation} if it depends on $C_1$, $C_2$, $k$, $\beta$, $\gamma$, $a$, and $\|D_i\|_{q_i}$.

A sufficient condition for a quasilinear form to satisfy H.1 and H.2 is that the quasilinear form is adapted to a reference Dirichlet form in the sense of Definition \ref{def:adapted form} below. In many applications, the adaptedness is easy to verify. The sufficiency is proved in Section \ref{ssec:adapted forms}.

\subsection{Lorentz spaces} 

For Borel measurable functions $f:X \to \R$, the Lorentz quasi-norms are defined as
\begin{align*}
\| f \|_{r,r_1} = \left( r_1 \int_0^{\infty} (s^r \mu(|f| \ge s))^{r_1/r} \frac{ds}{s} \right)^{\frac{1}{r_1}}
\end{align*}
for $r,r_1 \in (0,\infty)$, and
\begin{align*}
\| f \|_{r,\infty} = \sup_{s>0} \left\{ s \, \mu(|f| \ge s)^{1/r} \right\}, \quad \| f \|_{\infty,\infty} = \| f \|_{\infty}.
\end{align*}
Observe that $\| f \|_{r,r} = \| f \|_r$.
The Lorentz space $L^{r,r_1}$ is defined as the collection of Borel measurable functions $f:X \to \R$ with $\| f \|_{r,r_1} < \infty$.

For any $\sigma >0$, it holds
\begin{align} \label{eq:lorentz power}
\|f^{\sigma}\|_{r,r_1}^{\frac{1}{\sigma}} 
= \|f\|_{\sigma r,\sigma r_1}. 
\end{align}
Furthermore,
\begin{align} \label{eq:lorentz monotonicity}
\| f \|_{r,r_2} \le 2^{2/r_1} \| f \|_{r,r_1}, \quad \mbox{ for all } 0 < r_1 \le r_2 \le \infty,
\end{align}
see, e.g., \cite[(4.3)]{BCLSC95}.

The Lorentz-H\"older inequality (\cite[Theorem 3.5]{ONeil63})
\begin{align} \label{eq:lorentz-hoelder}
\| fg \|_{r,r_1} \le\| f \|_{a,a_1} \| g \|_{b,b_1}
\end{align}
holds whenever
\begin{align*}
\frac{1}{r} = \frac{1}{a} + \frac{1}{b}, \quad \frac{1}{r_1} = \frac{1}{a_1} + \frac{1}{b_1}.
\end{align*}
This and \eqref{eq:lorentz power} imply that
\begin{align} \label{eq:lorentz-hoelder f=g}
\| f \|_{r,r_1} \le\| f \|_{a,a_1}^{\sigma} \| f \|_{b,b_1}^{1-\sigma}
\end{align}
holds whenever
\begin{align*}
\frac{1}{r} = \frac{\sigma}{a} + \frac{1-\sigma}{b}, \quad \frac{1}{r_1} = \frac{\sigma}{a_1} + \frac{1-\sigma}{b_1}, \quad 0 \le \sigma \le 1.
\end{align*}

For any $r \ge 1$, we let $r'$ be its H\"older conjugate, and $\frac{r''}{2}$ be the H\"older conjugate of $\frac{r}{2}$. That is,
\[ \frac{1}{r} + \frac{1}{r'} = 1, \qquad \frac{1}{r} + \frac{1}{r''} = \frac{1}{2}. \] 

The next lemma is similar to \cite[Lemma 1]{AS67}. 

\begin{lemma} \label{lem:lorentz norm of u}
Let $w \in L^2(I \to \F(U)) \cap L^{\infty}(I \to L^2(U))$ and $\nu >2$.
Then $w$ is in $L^{2q'}(I \to L^{2r',2}(U))$ for all exponent pairs $(r',q')$ whose H\"older conjugates $(r,q)$ satisfy $1 - \frac{1}{q} - \frac{\nu}{2r} \ge \gamma \ge 0$.
 Moreover,
\begin{align*}
 \left(\int_I ||w||_{2r',2}^{2q'} dt \right)^{\frac{1}{q'}} 
& \le 
|I|^{\gamma} \left( \int_I \| w \|_{\frac{2\nu}{\nu-2},2}^2 dt \right)^{\frac{\nu}{2r}} 
\left( \sup_{t \in I} \|w\|_{2}^2 \right)^{1-\frac{\nu}{2r}},
\end{align*}
where the Lorentz quasi-norms are taken over $U$.
\end{lemma}

\begin{proof}
By the H\"older inequality and \eqref{eq:lorentz-hoelder f=g},
\begin{align*}
 \left(\int_I ||w||_{2r',2}^{2q'} dt \right)^{\frac{1}{q'}} 
& \le 
|I|^{\gamma} \left( \int_I  \| w \|_{2r',2}^{2/\sigma} dt \right)^{\sigma} \\
& \le 
|I|^{\gamma} \left( \int_I \left( \| w \|_{\frac{2\nu}{\nu-2},2}^{\sigma} \|w\|_{2}^{1-\sigma} \right)^{2/\sigma} dt \right)^{\sigma} \\
& \le 
|I|^{\gamma} \left( \int_I \| w \|_{\frac{2\nu}{\nu-2},2}^2 dt \right)^{\sigma} 
\sup_{t \in I} \|w\|_{2}^{2(1-\sigma)},
\end{align*}
provided that $\gamma \ge 0$, $0 \le \sigma \le 1$, and
\begin{align*}
\frac{1}{q'} = \sigma+\gamma, \quad \frac{1}{2r'} = \frac{\sigma}{\frac{2\nu}{\nu-2}} + \frac{1-\sigma}{2}.
\end{align*}
These relations imply that
$\gamma = \frac{1}{q'} - \sigma = 1- \frac{1}{q} - \frac{\nu}{2r}$ and $\sigma = \frac{\nu}{2r}$.
\end{proof}

We fix $\gamma \in [0,1)$ and define
\begin{align} \label{eq:||| norm}
||| u |||_{I \times U} := \sup_{r',q'} \left(\int_I \|u\|_{L^{2r',2}(U)}^{2q'} dt \right)^{\frac{1}{2q'}},
\end{align}
where the supremum is taken over all pairs $(r',q')$ whose H\"older conjugates $(r,q)$ satisfy 
\begin{align} \label{eq:gamma}
1 - \frac{1}{q} - \frac{\nu}{2r} \ge \gamma, \quad r \ge \frac{1}{1-\gamma}.
\end{align}
Note that since $\gamma < 1$, 
\begin{align*}
||| u |||_{I \times U} = \sup_{r'',q''} \left(\int_I \|u\|_{L^{r'',2}(U)}^{q''} dt \right)^{\frac{1}{q''}}
\end{align*} 
where the supremum is taken over all $(r'',q'')$ whose corresponding pair $(r,q)$ satisfies
\begin{align*} %\label{eq:gamma''}
1 - \frac{1}{q} - \frac{\nu}{2r} \ge \frac{\gamma}{2}, \quad r \ge \frac{2}{1-\gamma}.
\end{align*}
Since $\nu > 2$, \eqref{eq:gamma} implies that $r>1$. Also, $q >1 $, $r' \ge 1$, $r'' \ge 2$.

\subsection{Geometric hypotheses: weighted Sobolev inequality and weighed Poincar\'e inequality}

\begin{definition}
Let $\delta^* \le \delta' < \delta \le 1$.
A function $\psi:X \to [0,1]$ is a {\em cutoff function for $B_{\delta'}$ in $B_{\delta}$} if $\psi \in \F \cap \mathcal{C}_{\mbox{\tiny{c}}}(B_{\delta})$ and $\psi = 1$ in $B_{\delta'}$.
\end{definition}

\begin{definition} \label{def:wSI}
We say that the {\em weighted Sobolev inequality} holds for a non-negative function $f$ if there exist constants $k\ge 0$, $\nu>2$ and $C_{\mbox{\tiny{SI}}}$, $C_{SI0} \ge 1$ such that, for any $\delta^* \le \delta' < \delta \le 1$, there is a cutoff function $\psi=\psi_{\delta',\delta}$ for $B_{\delta'}$ in $B_{\delta}$ such that
\begin{align} \label{eq:wSI} \tag{\bf wSI}
\| f^{p/2} \psi \|_{\frac{2\nu}{\nu-2},2}^2 
\le  
|\delta-\delta'|^{-k} \left( C_{\mbox{\tiny{SI}}} \frac{p^2}{4} \int_{B_{\delta}} f^{p-2} \psi^2 d\Gamma(f) + C_{\mbox{\tiny{SI0}}}  \int_{B_{\delta}}  f^{p} d\mu \right),
\end{align}
for all $p \in \R$ with $f^{p/2} \in \F$.
The constants $C_{\mbox{\tiny{SI}}}$, $C_{SI0}$ may depend on $\delta^*$ but not on $\delta',\delta$.
\end{definition}

\begin{remark}
\begin{enumerate}
\item
The choice of the Lorentz parameter $r_1=2$ in \eqref{eq:wSI} is sharp, see \cite[Remark 4.5]{BCLSC95}.
\item 
There is no loss of generality in assuming that $k$ is the same exponent as in \eqref{eq:C_3}. 
\end{enumerate}
\end{remark}

\begin{definition}
We say that the {\em weighted Poincar\'e inequality} holds for $\log f$, where $f \in \F$ is a uniformly positive function, if there is a positive constant $C_{\mbox{\tiny{wPI}}}$ such that
\begin{align} \label{eq:weighted PI} \tag{\bf wPI}
 \int |\log f - (\log f)_B|^2 \psi^2 d\mu  \le 
\int C_{\mbox{\tiny{wPI}}}(\delta',\delta) \psi^2 f^{-2} d\Gamma(f),
\end{align}
where $(\log f)_B = \int \log f \psi^2 d\mu \big/ \int \psi^2 d\mu$ is the weighted mean of $f$ over $B$, and $\psi = \psi_{\delta',\delta}$.
\end{definition}

The Cacciopoli-type estimates and the mean value estimates rely only on the weighted Sobolev inequality. For the parabolic Harnack we need in addition the weighted Poincar\'e inequality which is used in the Lemma \ref{lem:log lemma}.

The next lemma is similar to \cite[Lemma 3]{AS67}. 

\begin{lemma} \label{lem:|||f|||}
Let $f \in L^2(I \to \F(B_1)) \cap L^{\infty}(I \to L^2(B_1))$. 
Suppose the weighted Sobolev inequality \eqref{eq:wSI} holds for $f(t)$ uniformly for all $t \in I$. 
Then for any $\delta^* \le \delta' < \delta \le 1$ and any $\sigma \ge 1$,
\begin{align*}
& \quad  |||f^{\sigma}|||_{I \times B_{\delta'}}^{\frac{2}{\sigma}} \\
& \le 
2 |I|^{\gamma}  \left(  \frac{C_{\mbox{\tiny{SI}}}}{|\delta-\delta'|^k}  \int_I\int  \psi^2 d\Gamma(f) dt +  \frac{C_{\mbox{\tiny{SI0}}}}{|\delta-\delta'|^k}  \int_I \int_{B_{\delta}} f^2 d\mu \, dt +  \sup_{t \in I} \| f \|^2_{L^2(B_{\delta})} \right),
\end{align*}
where $\psi = \psi_{\delta',\delta}$.
\end{lemma}

\begin{proof}
For any pair $(r',q')$ whose H\"older conjugates $(r,q)$ satisfy \eqref{eq:gamma}, 
the H\"older conjugates of $(\sigma r',\sigma q')$ also satisfy \eqref{eq:gamma}, and
\begin{align*}
||f^{\sigma}||_{L^{2q'}(I \to L^{(2r',2)})}^{\frac{1}{\sigma}} = ||f||_{L^{2\sigma q'}(I \to L^{(2\sigma r',2\sigma)})} \le 2 ||f||_{L^{2\sigma q'}(I \to L^{(2\sigma r',2)})},
\end{align*}
where the estimate is due to \eqref{eq:lorentz monotonicity}. Thus, by Lemma \ref{lem:lorentz norm of u} and \eqref{eq:wSI},
\begin{align*}
& \quad 
||f^{\sigma}||_{L^{2q'}(I \to L^{(2r',2)}(B_{\delta'}))}^{\frac{1}{\sigma}} \\
& \le 
2 |I|^{\gamma} |\delta-\delta'|^{-\frac{k\nu}{2\sigma r}} \left(  C_{\mbox{\tiny{SI}}} \int_I\int  \psi^2 d\Gamma(f) dt +  C_{\mbox{\tiny{SI0}}}  \int_I \int_{B_{\delta}} f^2 d\mu \, dt \right)^{\frac{\nu}{2\sigma r}} \\
& \quad \left( \sup_{t \in I} \| f \|^2_{L^2(B_{\delta})} \right)^{1 - \frac{\nu}{2\sigma r}}.
\end{align*}
Now apply Young's inequality and take the supremum over all pairs $(r',q')$ whose H\"older conjugates satisfy \eqref{eq:gamma}.
\end{proof}

\section{Mean value estimates} \label{sec:MVE}

\subsection{Chain rule in the time variable}\label{ssec:steklov chain rule}

\begin{lemma}
Let $u$ be a local weak subsolution of the heat equation for $\e_t$ in $Q^-$. Let $p \ge 2$.
Let $\chi:I^- \to [0,1]$ be any smooth function. Let $\psi \in \F_{\mbox{\tiny{c}}}(B_1)$.
Then, for almost every $a-a_1 < s_0 < t_0 < b'$,
\begin{align} \label{eq:steklov subsol estimate p>1}
\begin{split}
& \quad \int_X \mathcal H_n( u(t_0) )  \psi^2 d\mu 
- \int_X \mathcal H_n( u(s_0) )  \psi^2 d\mu\\
& \le
- \int_J \e_t( u(t), \mathcal H'_n(u(t)) \psi^2 ) \chi(t) \, dt
+ \int_J \int_X  \mathcal H_n(u(t)) \psi^2 \chi' \, d\mu \, dt.
\end{split}
\end{align}
where $J = (s_0, t_0)$.
\end{lemma}

\begin{proof}
For a real number $0 < h < b-b'$, let
\[  u_h(t) := \frac{1}{h} \int_t^{t+h} u(s) ds,\]
be the Steklov average of $u$ at $t <b'$. Here, the integral is a Bochner integral over functions that take values in the Banach space $(\F,\|\cdot\|)$. By definition, $u_h(t) \in \F$. 

Since $u$ is a local weak subsolution,
\begin{align*}
& \quad  \int_X \mathcal H_n( u_h(t_0) )  \psi^2 d\mu - \int_X \mathcal H_n( u(s_0) )  \psi^2 d\mu\\
& =
\int_J \frac{d}{d t} \left( \int_X \mathcal H_n( u_h(t) ) \psi^2 \chi \, d\mu \right) dt 
\\
& = 
\int_J \frac{1}{h} \int_X \big( u(t+h) - u(t) \big) \mathcal H'_n(u_h(t)) \psi^2 \chi \,  d\mu \, dt 
 + \int_J \int_X  \mathcal H_n(u_h(t)) \psi^2 \chi' \, d\mu \, dt \\
& \leq 
- \int_J \frac{1}{h} \int_t^{t+h}\e_s( u(s), \mathcal H'_n(u_h(t)) \psi^2 ) ds \, \chi(t) dt
+ \int_J \int_X  \mathcal H_n(u_h(t)) \psi^2 \chi' \, d\mu \, dt.
\end{align*}
By \cite[Theorem 9]{DieUhl77}, $u_h(t)$ converges to $u(t)$ in $\F$ as $h \to 0$, at almost every $t$. Since the $\F$-norm dominates the $L^2$-norm, it follows that
%Notice that (see \cite[Proof of Theorem 3.11, STEP 3]{LierlPHI}), 
\begin{align*}
\int_X \mathcal H_n( u_h(t) )  \psi^2 d\mu 
\longrightarrow 
\int_X \mathcal H_n( u(t) )  \psi^2 d\mu \quad \mbox{ at a.e. } t
\end{align*}
and
\begin{align*}
\int_J \int_X  \mathcal H_n(u_h) \psi^2 \chi' \, d\mu \, dt 
\longrightarrow 
\int_J \int_X  \mathcal H_n(u) \psi^2 \chi' \, d\mu \, dt
\end{align*}
as $h \to 0$ (passing to a subsequence if necessary).

It remains to show that $- \int_J \frac{1}{h} \int_t^{t+h}\e_s( u(s,\cdot), \mathcal H'_n(u_h(t,\cdot)) \psi^2 ) \chi(t) ds \, dt$ converges to $-\int_J \e_t( u(t,\cdot), \mathcal H'_n(u(t,\cdot)) \psi^2 )  \chi(t) dt$ as $h \to 0$. 
We have
\begin{align*}
\mathcal H'_n(u_h(t)) \psi^2  \longrightarrow  \mathcal H'_n(u(t)) \psi^2
\end{align*}
in $\F$ as $h \to 0$, at almost every $t$.
Hence, by the right-continuity of $\e_t$,
\begin{align*}
\left| \int_J \frac{1}{h} \int_t^{t+h}\e_s( u(s), [\mathcal H'_n(u_h(t)) - \mathcal H'_n(u(t))] \psi^2 ) ds \, \chi(t) dt \right| \longrightarrow 0,
\end{align*}
Applying \cite[Theorem 9]{DieUhl77} with $f(s) = \e_s( u(s), \mathcal H'_n(u(t)) \psi^2 )$, we see that
\begin{align*}
\int_J \frac{1}{h} \int_t^{t+h} \left| \e_s( u(s), \mathcal H'_n(u(t)) \psi^2 ) - \e_t(u(t), \mathcal H'_n(u(t)) \psi^2 ) \right| ds \, \chi(t) dt \longrightarrow 0.
\end{align*}
Indeed, $f$ is integrable due to the right-continuity of $\e_t$.
Combining the above and using the right-linearity of $\e_t$ completes the proof.
\end{proof}

\begin{lemma} \label{lem:steklov supsol}
Let $u$ be a non-negative locally bounded local weak supersolution of the heat equation for $\e_t$ in $Q^{\pm}$.  Let $p \le 1-\eta$.
Let $\chi:I^{\pm} \to [0,1]$ be any smooth function. Let $\psi \in \F_{\mbox{\tiny{c}}}(B_1)$.
Then, for any interval $J = (s_0,t_0) \subset I^{\pm}_1$,
\begin{align*} 
\begin{split}
& \quad \int_X \mathcal H( u_{\varepsilon}(t_0) )  \psi^2 d\mu 
- \int_X \mathcal H( u_{\varepsilon}(s_0) )  \psi^2 d\mu\\
& \ge
- \int_J \e_t( u(t), \mathcal H'(u_{\varepsilon}(t)) \psi^2 ) \chi(t) \, dt
+ \int_J \int_X  \mathcal H(u_{\varepsilon}(t)) \psi^2 \chi' \, d\mu \, dt.
\end{split}
\end{align*}
\end{lemma}

\begin{proof}
First consider the case when $\pm$ is $-$.
For a real number $0 < h < b-b'$, let
\[  (\bar u_{\varepsilon})_h(t) := \frac{1}{h} \int_t^{t+h} \bar u_{\varepsilon}(s) ds,\]
be the (upper) Steklov average of $\bar u_{\varepsilon}$ at $t <b'$.

Since $u$ is a local weak supersolution,
\begin{align*}
& \quad  \int_X \mathcal H( (u_{\varepsilon})_h(t_0) )  \psi^2 d\mu - \int_X \mathcal H( (u_{\varepsilon})_h(s_0) )  \psi^2 d\mu\\
& =
\int_J \frac{d}{d t} \left( \int_X \mathcal H( (u_{\varepsilon})_h(t) ) \psi^2 \chi \, d\mu \right) dt 
\\
& = 
\int_J \frac{1}{h} \int_X \big( u(t+h) - u(t) \big) \mathcal H'((u_{\varepsilon})_h(t)) \psi^2 \chi \,  d\mu \, dt 
 + \int_J \int_X  \mathcal H((u_{\varepsilon})_h(t)) \psi^2 \chi' \, d\mu \, dt \\
& \ge 
- \int_J \frac{1}{h} \int_t^{t+h}\e_s( u(s), \mathcal H'((u_{\varepsilon})_h(t)) \psi^2 ) ds \, \chi(t) dt
+ \int_J \int_X  \mathcal H((u_{\varepsilon})_h(t)) \psi^2 \chi' \, d\mu \, dt.
\end{align*}
By \cite[Theorem 9]{DieUhl77}, $u_h(t)$ converges to $u(t)$ in $\F$ as $h \to 0$, at almost every $t$. Since the $\F$-norm dominates the $L^2$-norm, and since $(\bar u_{\varepsilon})_h = u_h + \kappa + \varepsilon$  it follows that
\begin{align*}
\int_X \mathcal H( (u_{\varepsilon})_h(t) )  \psi^2 d\mu 
\longrightarrow 
\int_X \mathcal H( u_{\varepsilon}(t) )  \psi^2 d\mu \quad \mbox{ at a.e. } t
\end{align*}
and
\begin{align*}
\int_J \int_X  \mathcal H((u_{\varepsilon})_h) \psi^2 \chi' \, d\mu \, dt 
\longrightarrow 
\int_J \int_X  \mathcal H(u_{\varepsilon}) \psi^2 \chi' \, d\mu \, dt
\end{align*}
as $h \to 0$ (passing to a subsequence if necessary).

It remains to show that $- \int_J \frac{1}{h} \int_t^{t+h}\e_s( u(s), \mathcal H'((u_{\varepsilon})_h(t)) \psi^2 ) \chi(t) ds \, dt$ converges to $-\int_J \e_t( u(t), \mathcal H'(u_{\varepsilon}(t)) \psi^2 )  \chi(t) dt$ as $h \to 0$. 
We have
\begin{align*}
\mathcal H'((u_{\varepsilon})_h(t)) \psi^2  \longrightarrow  \mathcal H'(u_{\varepsilon}(t)) \psi^2
\end{align*}
in $\F$ as $h \to 0$, at almost every $t$.
Hence, by the right-continuity of $\e_t$,
\begin{align*}
\left| \int_J \frac{1}{h} \int_t^{t+h}\e_s( u(s), [\mathcal H'((u_{\varepsilon})_h(t)) - \mathcal H'((u_{\varepsilon})(t))] \psi^2 ) ds \, \chi(t) dt \right| \longrightarrow 0,
\end{align*}
Applying \cite[Theorem 9]{DieUhl77} with $f(s) = \e_s( u(s), \mathcal H'(u_{\varepsilon}(t)) \psi^2 )$, we see that
\begin{align*}
\int_J \frac{1}{h} \int_t^{t+h} \left| \e_s( u(s), \mathcal H'(u_{\varepsilon}(t)) \psi^2 ) - \e_t(u(t), \mathcal H'(u_{\varepsilon}(t)) \psi^2 ) \right| ds \, \chi(t) dt \longrightarrow 0.
\end{align*}
Indeed, $f$ is integrable due to the right-continuity of $\e_t$.
Combining the above and using the right-linearity of $\e_t$ completes the proof in the case $p \in (-\infty,0)$.

In the case when $\pm$ is $+$, we use the (lower) Steklov average of $\bar u_{\varepsilon}$ at $t> a'$, defined as
\[  (\bar u_{\varepsilon})_h(t) := \frac{1}{h} \int_{t-h}^t \bar u_{\varepsilon}(s) ds,\]
where $0 < h < a'-a$. Then the proof is as in the previous case.
\end{proof}

\subsection{Estimates for local weak subsolutions} 

\begin{theorem}[Cacciopoli-type inequality for subsolutions] \label{thm:estimate subsol p>2}
Let $u$ be a local weak subsolution of the heat equation for $\e_t$ in $Q^-$.  Suppose H.1a holds for $u$. Then, for any $p \ge 2$, 
%Suppose $\int_{I^-_{\delta}} \int_{B_{\delta}} u^p d\mu dt < \infty$. 
\begin{align} \label{eq:subsol p>2}
\begin{split}
& \quad \frac{1}{2} \sup_{t \in I^-_{\delta'}} \int \bar u^p \psi^2 d\mu + a \frac{p^2}{4} \int_{I^-_{\delta'}} \int \bar u^{p-2} \psi^2 d\Gamma(u) dt \\
& \le 
p^2 C_1 |||\bar u^{\frac{p}{2}} \psi |||_{I_{\delta}^- \times B_{\delta}}^2 
+ \left( p^{\beta+1} C_2 + 2C_3 \right) |\delta'-\delta|^{-k}  \int_{I_{\delta}^-} \int \bar u^p \psi \, d\mu \, dt + (p-1) \int \kappa^p \psi^2 d\mu,
\end{split}
\end{align}
provided that the right hand side is finite.

If, in addition, H.1b holds for $u$ and all $p \in (1+\eta,2)$, then \eqref{eq:subsol p>2} also holds for these values of $p$.
\end{theorem}

\begin{proof}
Let $\chi = \chi_{\delta',\delta}$ be a smooth function of the time variable $t$ such that $0 \leq \chi \leq 1$,  $\chi = 0 \textrm{ in } (-\infty, a - a_{\delta})$, $\chi = 1$ in $(a - a_{\delta'}, \infty)$ and $|\chi'| \le 2|a_{\delta}-a_{\delta'}|^{-1}$. 

From H.1a and \eqref{eq:steklov subsol estimate p>1}, we get
\begin{align*}
\begin{split}
&\quad \int_X \mathcal H_n( u(t_0) )  \psi^2 d\mu + \frac{a}{2} \int_{I_{\delta'}^-} \int \bar u_n^{p-2} \psi^2 d\Gamma(u) dt 
+ \frac{p-2}{4} a \int_{I_{\delta'}^-} \int_{\{u\le n\}} \bar u_n^{p-2} \psi^2 d\Gamma(u) \\
& \le
p C_1 |||\bar u \bar u_n^{\frac{p-2}{2}} \psi |||_{I_{\delta}^- \times B_{\delta}}^2 
 + p^{\beta} C_2 |\delta'-\delta|^{-k} \int_{I_{\delta}^-} \int \bar u^2 \bar u_n^{p-2} \psi \, d\mu \, dt 
+ \int_{I_{\delta}^-} \int_X  \mathcal H_n(u) \psi^2 \chi' \, d\mu \, dt.
 \end{split}
\end{align*}
We multiply each side by $p$, let $n \to \infty$, and take the supremum over $t_0 \in I^-_{\delta'}$. 
By Young's inequality,
\begin{align*}
\lim_{n \to \infty} p \H(u_n) 
= \bar u^p - p \bar u \kappa^{p-1} + (p-1) \kappa^p
\ge \frac{1}{2} \bar u^p - (p-1) \kappa^p.  
\end{align*}
Hence,
\begin{align*}
&\quad \int_X \left( \frac{1}{2} \bar u(t_0)^p -  (p-1) \kappa^p \right) \psi^2 d\mu 
+ a \frac{p^2}{4} \int_{I_{\delta'}^-} \int_X \bar u^{p-2} \psi^2 d\Gamma(u) \\
& \le
p^2 C_1 |||\bar u^{\frac{p}{2}} \psi |||_{I_{\delta}^- \times B_{\delta}}^2  +  \int_{I_{\delta}^-} \int \left( p^{\beta+1} C_2 |\delta'-\delta|^{-k} + \chi' \right) \bar u^p  \psi \, d\mu \, dt,
\end{align*}
where we have used that $\kappa \le \bar u$. Finally, apply \eqref{eq:C_3} to estimate $|\chi'|$.
\end{proof}

\subsection{Local boundedness and $L^{p,\infty}$ mean value estimates for $p\ge2$}

Define $B_{\delta}$, $I^-$, $I^+$, $I^-_{\delta}$, $I^+_{\delta}$, $Q^-_{\delta}$, $Q^+_{\delta}$ as in Section \ref{ssec:structural hypotheses}. 

\begin{lemma} \label{lem:sup_t estimate}
 Let $u$ be a local weak subsolution to the heat equation for $\e_t$ in $Q^-$. Suppose \eqref{eq:wSI} holds for $\bar u(t)$ uniformly for all $t \in I^-$. Then 
\begin{align*}
 \sup_{t \in I_{\delta'}^-} \| \bar u \psi \|_{2}^2  
+  \int_{I_{\delta'}^-} \int \psi^2 d\Gamma(u) dt  \le 
C \int_{I_{\delta}^-}  \int_{B_{\delta}}  \bar u^2 d\mu  \, dt ,
\end{align*}
where $\psi = \psi_{\delta',\delta}$. The constant $C \in (0,\infty)$ depends only on $\beta$, $\gamma$, $a$, $C_1$, and upper bounds for $(C_2 + C_3 + C_{\mbox{\tiny{SI}}}) |\delta-\delta'|^{-k}$, $\frac{C_{\mbox{\tiny{SI0}}}}{C_{\mbox{\tiny{SI}}}}$, and $|I_{\delta}^-|$.
\end{lemma}

\begin{proof}
Our proof follows \cite[Section 3]{AS67}.
Let $\chi:\R \to [0,1]$ be a smooth function with $\chi = 0$ on $(-\infty,a-a_{\delta})$ and $\chi =1$ on $(a-a_{\delta'},\infty)$. Due to \eqref{eq:C_3}, we may assume that $|\chi'| \le 2C_3 |\delta-\delta'|^{-k}$. We choose $s_0 \in (a-a_1,a-a_{\delta})$ and let $s_n = s_0 + nL$ for some $L>0$ given below.  
Let $J = (s_n,s_{n+1})$, assuming that $s_{n+1} \in \cap I_{\delta}^-$.

Let $u^+ = \max(u,0)$ and
\[ X(t) = \frac{1}{2} \int u^+(t)^2 \psi^2 \chi(t) \, d\mu. \]
As in the proof of Theorem \ref{thm:estimate subsol p>2} with $p=2$, we have for almost every $t \in J$,
\begin{align*}
& \quad X(t) - X(s_n) + a \int_J \int \psi^2 d\Gamma(u) \chi \, dt  \\
&  \le 
4 C_1 |||\bar u \psi \chi^{1/2} |||_{J \times B_{\delta}}^2 
+  \left( 2^{\beta+1}C_2 + 2C_3 \right)  |\delta'-\delta|^{-k} \int_J \int  \bar u^2  \psi \, \chi \, d\mu \, dt + \int \kappa^2 \psi^2 d\mu.
\end{align*}
Repeating the proof of Lemma \ref{lem:|||f|||} with $\sigma=1$ and $f = \bar u\psi \chi^{1/2}$,
\begin{align*}
& \quad 
|||\bar u \psi \chi^{1/2} |||_{J \times B_{\delta}}^2 \\
%& \le 
%2 L^{\gamma} \left(  C_{\mbox{\tiny{SI}}}(\delta',\delta) \int_J\int  \psi^2 d\Gamma(f) dt +  C_{\mbox{\tiny{SI0}}} |\delta-%\delta'|^{-k} \int_J \int_{B_{\delta}} f^2 d\mu \, dt %\right)^{\frac{\nu}{2p}} \left( \sup_{t \in J} \| f \|%^2_{L^2(B_{\delta})} \right)^{1 - \frac{\nu}{2p}} \\
& 
\le 
2 L^{\gamma} \left(  \frac{C_{\mbox{\tiny{SI}}}}{|\delta-\delta'|^k} \int_J\int_{B_{\delta}}  \psi^2 d\Gamma(u) \chi dt +  \frac{C_{\mbox{\tiny{SI0}}}}{|\delta-\delta'|^k} \int_J \int_{B_{\delta}} \bar u^2 d\mu \, \chi dt +  \sup_{t \in J} \| \bar u \psi \|_2^2 \chi \right).
\end{align*}
 
Now we choose
\[  L := \left( \frac{a |\delta-\delta'|^k}{48 C_1 \cdot  C_{\mbox{\tiny{SI}}}} \right)^{\frac{1}{\gamma}}. \]
Then, for almost every $t \in J$,
\begin{align} \label{eq:X(t_0)}
\begin{split}
& \quad X(t)
+  \frac{a}{2} \int_J \int  \psi^2 d\Gamma(u) \chi dt \\
& \le
X(s_n) +  \frac{a}{4} \sup_{t \in J} \|  u^+ \psi \|_2^2 \chi 
+ \left( \frac{a C_{\mbox{\tiny{SI0}}}}{4 C_{\mbox{\tiny{SI}}}} + \left( 2^{\beta+1}C_2 + 2C_3 \right)  |\delta'-\delta|^{-k} \right)   \int_J  \int_{B_{\delta}}  \bar u^2 d\mu \, \chi dt \\
& \quad + 2 \int \kappa^2 \psi^2 d\mu.
\end{split}
\end{align}
Disregarding the non-negative integral on the left hand side of \eqref{eq:X(t_0)}, rearranging, and taking supremum over all $t \in J$, 
\begin{align*}
\frac{1}{4} \sup_{t \in J} \| u^+ \psi \|_{2}^2  \chi 
\le X(s_n) +
\left( \frac{a C_{\mbox{\tiny{SI0}}}}{4 C_{\mbox{\tiny{SI}}}} +  \frac{2^{\beta+1}C_2 + 2C_3 }{  |\delta'-\delta|^k } + 2 \frac{1_{\{\kappa>0\}}}{L} \right)  \int_J \int_{B_{\delta}}  \bar u^2 d\mu  \, \chi dt.
\end{align*}
Iterating over the time-intervals $(s_n,s_{n+1})$, we obtain
\begin{align} \label{eq:sup_t <}
\sup_{t \in {I_{\delta}^-}} \| u^+ \psi \|_2^2 \chi  
 \le
2^{1+|{I_{\delta}^-}|/L} K + X(s_0),
\end{align}
where
\[  K := \left( \frac{a C_{\mbox{\tiny{SI0}}}}{4 C_{\mbox{\tiny{SI}}}} +  \frac{2^{\beta+1}C_2 + 2C_3 }{  |\delta'-\delta|^k } + 2 \frac{1_{\{\kappa>0\}}}{L} \right)  \int_{I_{\delta}^-} \int_{B_{\delta}}  \bar u^2 d\mu  \, \chi dt. \]
By the choice of $\chi$ and $s_0$, we have $X(s_0)=0$.

Putting \eqref{eq:sup_t <} into \eqref{eq:X(t_0)}, using $X(t) \ge 0$, and summing over all $(s_n,s_{n+1})$, we get
\begin{align*} 
\frac{a}{2} \int_{I_{\delta}^-} \int \psi^2 d\Gamma(u) \chi \, dt 
 \le 
 K + \left(\frac{a}{4} + 1 \right) (2^{1+|{I_{\delta}^-}|/L} K ).
\end{align*}
Simplifying,
\begin{align*} 
\int_{I_{\delta'}^-} \int \psi^2 d\Gamma(u) dt 
 \le 
 \frac{2}{a} (2^{2+|{I_{\delta}^-}|/L} K ).
\end{align*}
\end{proof}

\begin{lemma}[Gain of integrability] \label{lem:gain of integrability} 
Let $u$ be a local weak subsolution of the heat equation for $\e_t$ in $Q^-$. Suppose \eqref{eq:wSI} holds for $\bar u(t)$ uniformly for all $t \in I^-$. Then
\begin{align*}
|||\bar u^{\sigma}|||_{I_{\delta'}^- \times B_{\delta'}} < \infty 
\end{align*}
for all $\sigma >1$ and all $\delta' \in [\delta^*,1)$. 
\end{lemma}

\begin{proof}
Due to Lemma \ref{lem:sup_t estimate}, $\bar u$ is in $L^{\infty}(I_{\delta}^- \to L^2(B_{\delta}))$. Therefore we can apply Lemma \ref{lem:|||f|||} to get
\begin{align} \label{eq:HoelderSobolev integrated}
\begin{split}
& \quad |||\bar u^{\sigma}|||_{I_{\delta}^- \times B_{\delta'}}^{\frac{2}{\sigma}} \\
& \le  
2 |I_{\delta}^-|^{\gamma} \left(  \frac{C_{\mbox{\tiny{SI}}}}{|\delta'-\delta|^k}\int_{I_{\delta}^-} \int_{B_{\delta}}  \psi^2 d\Gamma(u) dt +  \frac{C_{\mbox{\tiny{SI0}}}}{|\delta'-\delta|^k} \int_{I_{\delta}^-} \int \bar u^2 d\mu \, dt + \sup_{t \in I_{\delta}^-} \| \bar u \|_{L^2(B_{\delta})}^2 \right),
\end{split}
\end{align}
where $\psi = \psi_{\delta',\delta}$.
The right hand side is finite by Lemma \ref{lem:sup_t estimate}.
\end{proof}

\begin{theorem}[Mean value estimate for subsolutions] \label{thm:MVE subsol p>2}
 Let $u$ be a local weak subsolution of the heat equation for $\e_t$ in $Q^-$. Suppose H.1a holds for $u$ and the weighted Sobolev inequality \eqref{eq:wSI} holds for $\bar u(t)$ uniformly for all $t \in I^-$.
Let $p\ge 2$.
Then there exists a positive constant $C'=C'(\nu,\beta,k)$ such that, for all $\delta^* \le \delta' < \delta \le 1$, 
\begin{align} \label{eq:MVE subsol p>2}
 \sup_{Q^-_{\delta'}} \{ \bar u^p \} 
& \le  
\frac{C' A_0^{\frac{\nu+2}{2}}}{ |\delta-\delta'|^{k(\nu+2)}}
|||\bar u^{\frac{p}{2}}|||_{I_{\delta}^- \times B_{\delta}}^2.
\end{align}
where
\[ A_0 :=  \frac{ 4|I_{\delta}^-|^{\gamma+1} C_{\mbox{\tiny{SI}}} p^{\beta+1}}{a |\delta-\delta'|^{2k}} \left(\frac{C_1 + 1_{\{\kappa>0\}}}{|I_{\delta}^-|} + \frac{ a C_{\mbox{\tiny{SI0}}}}{ C_{\mbox{\tiny{SI}}}} + C_2 + C_3 \right). \]

If, in addition, H.1b holds for $p \in (1+\eta,2)$, then \eqref{eq:MVE subsol p>2} also holds for these values of $p$.
\end{theorem}

\begin{proof}
Let $\chi$ be a smooth function of the time variable $t$ such that $0 \leq \chi \leq 1$,  $\chi = 0 \textrm{ in } (-\infty, a - a_{\delta})$, $\chi = 1$ in $(a - a_{\delta'}, \infty)$ and $|\chi'| \le 2|a_{\delta}-a_{\delta'}|^{-1}$.

Set $\hat\delta_i = (\delta - \delta') 2^{-i-1}$ so that $\sum_{i=0}^{\infty} \hat\delta_i = \delta - \delta'$. Set also $\delta_0 = \delta$, $\delta_{i+1} = \delta_i - \hat\delta_i = \delta - \sum_{j=0}^i \hat\delta_j$.

Let $\theta =\frac{\nu+2}{\nu}$. 
Let $\psi_i = \psi_{\delta_i,\delta_{i+1}}$ be the cutoff function for $B_{\delta_{i+1}}$ in $B_{\delta_i}$ that is given by \eqref{eq:wSI}. 

As in the proof of Lemma \ref{lem:|||f|||} but with  $\bar u^{p \theta^{i}/2} \psi$ in place of $f$ and $\sigma = \theta$, and then applying Theorem \ref{thm:estimate subsol p>2}, we get
\begin{align*} 
& \|  \bar u^{p \theta^{i+1}/2} \psi_i \|_{L^{2q'}(I_{\delta_i}^- \to L^{(2r',2)}(B_{\delta_i}))}^{\frac{2}{\theta}}  \\
& \le 2 |I_{\delta_i}^-|^{\gamma} \left( \frac{C_{\mbox{\tiny{SI}}}}{|\delta_i - \delta_{i+1}|^k} \frac{(p\theta^i)^2}{4} \int_{I_{\delta_i}^-} \int_{B_{\delta_i}} \bar u^{p \theta^i-2}\psi^2 d\Gamma(u) dt 
 + \frac{C_{\mbox{\tiny{SI0}}}}{|\delta_i - \delta_{i+1}|^k} \int_{I_{\delta_i}^-} \int_{B_{\delta_i}}  \bar u^{p \theta^{i}} d\mu \, dt \right)^{\frac{\nu}{2r}} \\
& \quad \left( \sup_{t \in I_{\delta_i}^-} \| \bar u^{p \theta^{i}/2} \|_{L^2(B_{\delta_i})}^2 \right)^{1-\frac{\nu}{2r}} \\
& \le 
2|I_{\delta_i}^-|^{\gamma} \left(\frac{ C_{\mbox{\tiny{SI}}}}{a |\delta_i-\delta_{i+1}|^k} + 2 \right) \bigg[
(p\theta^i)^2 
C_1 |||\bar u^{p \theta^{i}/2}|||_{I_{\delta_i}^- \times B_{\delta_i}}^2 \\
& \quad  + \left(\frac{(p\theta^i)^{\beta+1} C_2 + 2C_3 }{ |\delta_i-\delta_{i+1}|^k} + \frac{a C_{\mbox{\tiny{SI0}}}}{C_{\mbox{\tiny{SI}}}} \right) \int_{I_{\delta_i}^-} \int \bar u^{p\theta^i} \psi_i d\mu\, dt 
+ (p\theta^i-1) \int \kappa^{p\theta^i} \psi_i^2 d\mu\bigg].
\end{align*}
By H\"older's inequality,
\begin{align} \label{eq:hoelder application to |||}
\begin{split}
\int_{I_{\delta_i}^-} \int \bar u^{p\theta^i}  d\mu \, dt
& \le 
\|1\|_{L^1(I_{\delta_i}^- \to L^{\infty}(B_{\delta_i}))} 
\|\bar u^{p\theta^i}\|_{L^{\infty}(I_{\delta_i}^- \to L^1(B_{\delta_i}))} \\
& \le 
|I_{\delta_i}^-| \cdot |||\bar u^{\frac{p\theta^i}{2}} |||_{I_{\delta_i}^- \times B_{\delta_i}}^2.
\end{split}
\end{align}
Similarly, by H\"older's inequality  and the fact that $\kappa \le \bar u$,
\begin{align*}
\kappa^{p\theta^i} \int \psi_i^2 d\mu \le |||\bar u^{\frac{p\theta^i}{2}} |||_{I_{\delta_i}^- \times B_{\delta_i}}^2.
\end{align*}
Combining the above estimates and using that $\psi_i = 1$ on $B_{\delta_{i+1}}$,
\begin{align*} 
& \quad
 |||  \bar u^{p \theta^{i+1}/2} \psi_i |||_{I_{\delta_{i+1}}^- \times B_{\delta_{i+1}}}^{\frac{2}{\theta}}  \\
&  \le 
 \frac{ 4|I_{\delta_i}^-|^{\gamma+1} C_{\mbox{\tiny{SI}}} p^{\beta+1}}{a |\delta-\delta'|^{2k}} C^i  \left(\frac{C_1 + 1_{\{\kappa>0\}}}{|I_{\delta_i}^-|} + \frac{ a C_{\mbox{\tiny{SI0}}}}{ C_{\mbox{\tiny{SI}}}} + C_2 + C_3 \right)
 |||\bar u^{p \theta^{i}/2}|||_{I_{\delta_i}^- \times B_{\delta_i}}^2
\end{align*}
where $C$ depends only on $\theta$, $\beta$ and $k$.
Iterating the above inequality,
\begin{align*} 
& \quad |||  \bar u^{p \theta^{i+1}} |||_{I_{\delta_{i+1}}^- \times B_{\delta_{i+1}}}^{2\theta^{-(i+1)}}   \le 
C^{ \sum j \theta^{-j} } (A_0 |\delta-\delta'|^{-2k})^{ \sum \theta^{-j} } |||u^{\frac{p}{2}}|||_{I_{\delta}^- \times B_{\delta}}^2 \\
\end{align*}
where the sums are over $j = 0,1,\ldots,i$.
 Letting $i$ tend to infinity, we obtain
\begin{align*}
 \sup_{Q^-_{\delta'}} \{ \bar u^p \} 
& \le  
C' (A_0 |\delta-\delta'|^{-2k})^{\frac{\nu+2}{2}} 
|||\bar u^{\frac{p}{2}}|||_{I_{\delta}^- \times B_{\delta}}^2.
\end{align*}
This proves \eqref{eq:MVE subsol p>2} in the case $p \ge 2$.
Now Theorem \ref{thm:u loc bounded} already follows. In the case $1+\eta < p < 2$, the assertion can be proved in the same way as above, except that we use Theorem \ref{thm:u loc bounded} instead of Lemma \ref{lem:gain of integrability} to verify that the right hand side of \eqref{eq:subsol p>2} is finite.
\end{proof}

\begin{theorem}[Local boundedness] \label{thm:u loc bounded}
Under the same hypotheses as in Theorem \ref{thm:MVE subsol p>2}, any non-negative local weak subsolution $u$ of the heat equation for $\e_t$ is locally bounded. Moreover, if $u$ is a local weak solution of the heat equation for $\e_t$ and the hypotheses in Theorem \ref{thm:MVE subsol p>2} hold for both $u$ and $-u$, then $u$ is locally bounded.
\end{theorem}

\begin{proof}
In the proof of Theorem \ref{thm:MVE subsol p>2}, we have shown that for any local weak subsolution $u$, $max(u,0)$ is locally bounded. If $u$ is a weak solution, then the same reasoning applies to $-u$.
\end{proof}

\subsection{Estimates for local weak supersolutions}

Let $\varepsilon \in(0,1)$ and recall that $\bar u_{\varepsilon} := u + \kappa +\varepsilon$.
\begin{lemma}[Cacciopoli-type inequality supersolutions] \label{lem:estimate supsol}
Let $u$ be a non-negative locally bounded local weak supersolution of the heat equation for $\e_t$ in $Q^{\pm}$.
Suppose H.1b holds for $u$. 
Then for any $p \in (-\infty,0) \cup (0,1-\eta)$,
\begin{align} \label{eq:supsol}
\begin{split}
& \quad \sup_{t \in I^{\pm}_{\delta'}} \int \bar u_{\varepsilon}^p \psi^2 d\mu + a \eta \frac{p^2}{4}\int_{I^{\pm}_{\delta'}} \int \bar u_{\varepsilon}^{p-2} \psi^2 d\Gamma(u) dt  \\
& \le 
(1+p^2) C_1 |||  \bar u_{\varepsilon}^{\frac{p}{2}}|||_{I_{\delta}^{\pm} \times B_{\delta}}^2 + ((1 +|p|^{\beta+1})C_2 + 2C_3) |\delta'-\delta|^{-k}\int_{I_{\delta}^{\pm}} \int  \bar u_{\varepsilon}^{p} \psi d\mu \, dt,
\end{split}
\end{align}
where $\psi = \psi_{\delta',\delta}$.
Here, the superscript $\pm$ is $+$ when $p\in (0,1-\eta)$ and $-$ when $p \in (-\infty,0)$.
\end{lemma}

\begin{proof}
In the case $p \in (-\infty,0)$, we let $\chi:\R \to [0,1]$ be a smooth function with $\chi = 0$ on $(-\infty,a-a_{\delta})$, $\chi=1$ on $(a-a_{\delta'},\infty)$, and $|\chi'| \le 2|a_{\delta} - a_{\delta'}|^{-1}$. 
Let $s_0 \in (a - a_1,a-a_{\delta})$, $t_0 \in I_{\delta'}^-$ and set $J = (s_0, t_0)$. 
By Lemma \ref{lem:steklov supsol},
\begin{align*}
& \quad \int_X \mathcal H( u_{\varepsilon}(t_0) )  \psi^2 d\mu \\
& \ge
- \int_J \e_t( u(t), \mathcal H'( u_{\varepsilon}(t)) \psi^2 ) \chi(t) \, dt
+ \int_J \int_X  \mathcal H( u_{\varepsilon}(t)) \psi^2 \chi' \, d\mu \, dt.
\end{align*}

In the case $p \in (0,1-\eta)$, we let $\chi:\R \to [0,1]$ be a smooth function with  $\chi = 0$ in $(b + a_{\delta}, \infty)$, $\chi = 1$ in $(-\infty, b + a_{\delta'})$ and $|\chi'| \le 2|a_{\delta} - a_{\delta'}|^{-1}$.
Let $s_0 \in (b + a_{\delta},b+a_1)$, $t_0 \in I_{\delta}^+$ and set $J = (t_0, s_0)$. 
By Lemma \ref{lem:steklov supsol}
\begin{align*}
& \quad \int_X \mathcal H( u_{\varepsilon}(t_0) )  \psi^2 d\mu \\
& \le
 \int_J \e_t( u(t), \mathcal H'( u_{\varepsilon}(t)) \psi^2 ) \chi(t) \, dt
- \int_J \int_X  \mathcal H( u_{\varepsilon}(t)) \psi^2 \chi' \, d\mu \, dt.
\end{align*}

In either case, we get from the above inequalities and H.1b that
\begin{align*}
\begin{split}
&\quad p \int_X \mathcal H( u_{\varepsilon}(t_0) )  \psi^2 d\mu + \frac{|p| |p-1|}{4} a\int_{I^-_{\delta'}} \int \bar  u_{\varepsilon}^{p-2} \psi^2 d\Gamma(u) \chi dt \\
& \le
|p| \int_J \bigg[ \e_t( u, \mathcal H'(u_{\varepsilon}) \psi^2 ) +  \frac{|p-1|}{4} a \int  \bar u_{\varepsilon}^{p-2} \psi^2 d\Gamma(u) \bigg] \chi dt
 + p \int_J \int_X  \mathcal H(u_{\varepsilon}) \psi^2 |\chi'| \, d\mu \, dt \\
& \le
(1+p^2) C_1 |||\bar u_{\varepsilon}^{\frac{p}{2}} \psi|||_{J \times B_{\delta}}^2  
+ ((1+ |p|^{\beta+1})C_2 + 2C_3) |\delta'-\delta|^{-k} \int_J \int \bar u_{\varepsilon}^p \psi d\mu\, dt.
 \end{split}
\end{align*}
Applying H\"older's inequality as in \eqref{eq:hoelder application to |||}, we get
\begin{align*}
 \int_J \int \bar u_{\varepsilon}^p \psi d\mu \, dt
& \le 
 |J| \cdot ||| \bar u_{\varepsilon}^{\frac{p}{2}}|||_{I_{\delta}^{\pm} \times B_{\delta}}^2.
\end{align*}
Taking the supremum over $t_0 \in I^{\pm}_{\delta'}$ proves \eqref{eq:supsol} with $a \frac{|p-1|}{|p|}$ in place of $a \eta$. Note, however, that $\frac{|p-1|}{|p|} \ge \eta$. 
\end{proof}

The next theorem can be proved analogously to the proof of Theorem \ref{thm:MVE subsol p>2}, by applying Lemma \ref{lem:estimate supsol} instead of Theorem \ref{thm:estimate subsol p>2}.

\begin{theorem}[Mean value estimate for supersolutions] \label{thm:MVE supsol}
Let $u$ be a non-negative locally bounded local weak supersolution of the heat equation for $\e_t$ in $Q^{\pm}$.  
 Suppose H.1b holds for $u$ and the weighted Sobolev inequality \eqref{eq:wSI} holds for $\bar u_{\varepsilon}(t)$ uniformly for all $t \in I^{\pm}$.
Then there is a positive constant $C'=C'(\nu,\beta,k,\eta)$ such that the following holds for all $\delta^* \le \delta' < \delta \le 1$.
\begin{align*}
 \sup_{Q^{\pm}_{\delta'}} \{ \bar u_{\varepsilon}^p \}  
& \le  
\frac{C' A_0^{\frac{\nu+2}{2}}}{|\delta-\delta'|^{k(\nu+2)}}
|||\bar u_{\varepsilon}^{\frac{p}{2}}|||_{I_{\delta}^{\pm} \times B_{\delta}}^2,
\end{align*}
where
\[ A_0 :=  \frac{ 4|I_{\delta}^{\pm}|^{\gamma+1} C_{\mbox{\tiny{SI}}} (1+|p|)^{\beta+1}}{a\eta |\delta-\delta'|^{2k}} \left(\frac{C_1 }{|I_{\delta}^{\pm}|} + \frac{ a\eta C_{\mbox{\tiny{SI0}}}}{ C_{\mbox{\tiny{SI}}}} + C_2 + C_3 \right). \]
Here, the superscript $\pm$ is $+$ when $p\in (0,1-\eta)$ and $-$ when $p \in (-\infty,0)$.
\end{theorem}

\section{Proof of the parabolic Harnack inequality} \label{sec:PHI}

\subsection{The abstract lemma of Bombieri - Giusti}

The following lemma extends the ``abstract John-Nirenberg inequality" that was first proved by Bombieri and Giusti \cite[Theorem 4]{BomGiu72}. Our proof closely follows \cite[Lemma 2.2.6]{SC02}.

We will write $d\bar \mu = d\mu \times dt$. 
\begin{lemma} \label{lem:Bombieri}
Let $k_1,k_2  \ge 0$, $\eta \in (0,1)$, $C \in (0,\infty)$.
Let $f$ be a non-negative measurable function on $I_1^{\pm} \times B_1$ which satisfies
\begin{align*} 
& \quad \sup_{I^{\pm}_{\delta'} \times B_{\delta'}} f^p
\le   \frac{A_1}{ (\delta - \delta')^{k_1}} |||f^{\frac{p}{2}}|||_{I^{\pm}_{\delta} \times B_{\delta}}^2,
\end{align*}
for all $\delta^* \leq \delta' < \delta < 1$, $0 < p < 1-\eta$.
Suppose further that 
\begin{align} \label{eq:log f > lambda}
 \bar\mu \left(Q_1^{\pm} \cap \{ \log f > \lambda \}\right) \le \frac{A_2}{(\delta - \delta')^{k_2}} \frac{\bar \mu(Q_1^{\pm})}{\lambda}, \qquad \forall \lambda > 0.
\end{align}
Then there is a constant $A_3 \in [1,\infty)$, depending only on $\delta^*, \eta, \gamma, k_1, k_2,  A_1, A_2$, such that
\[  \sup_{I_{\delta^*} \times B_{\delta^*}} f^p \le A_3. \]
\end{lemma}

\begin{proof}
If $(r',r_1')$ have H\"older conjugates $(r,r_1)$ satisfying \eqref{eq:gamma}, then $r',r_1' \le \frac{1}{\gamma}$. Therefore, at the expense of multiplying $A_1(\delta - \delta')^{-k_1}$ by $(|I_1^{\pm}| \mu(B_1))^{\frac{\gamma}{2}}$, we may assume that $|I_1^{\pm}| \mu(B_1) = 1$. Because $|I_1^{\pm}| \mu(B_1) = 1$, increasing the exponent $r$ increases the $L^r$ norm and the $L^{r,\infty}$ quasi-norm, so
\begin{align*}
||| f^{\frac{p}{2}} |||_{I^{\pm}_{\delta} \times B_{\delta}}^2 \le \| f^{\frac{p}{2}} \|_{L^{2/\gamma}\left(I^{\pm}_{\delta} \to L^{2/\gamma,2}(B_{\delta})\right)}^2.
\end{align*}
For each Lorentz space $L^{r,r_1}$ there is a constant constant $K(r,r_1)>0$ such that the quasi-norm satisfies
\begin{align} \label{eq:triangle ineq}
\|u+v\|_{r,r_1} \le K(r,r_1) \left( \|u\|_{r,r_1}  + \|v\|_{r,r_1} \right)
\end{align}
for all $u,v \in L^{r,r_1}$. 
Define
\[ \phi = \phi(\delta) :=  \sup_{I_{\delta} \times B_{\delta}} f. \]
Decomposing $I^{\pm}_{\delta} \times B_{\delta}$ into the sets where $\log f > \frac{1}{2}\log(\phi)$ and where $\log f \leq \frac{1}{2}\log(\phi)$, we get 
from \eqref{eq:triangle ineq} and \eqref{eq:log f > lambda} that
\begin{align*}
& \quad \|f^\frac{p}{2}\|_{L^{2/\gamma}(I^{\pm}_{\delta} \to L^{2/\gamma,2}(B_{\delta}))}^2 \\
& \le
K \sup_{I^{\pm}_{\delta} \times B_{\delta}} f^p
\|1_{\{f > \phi^{1/2}\}} \|_{L^{2/\gamma}(I^{\pm}_{\delta} \to L^{2/\gamma,2}(B_{\delta}))}^2 
+ K \phi^{p/2} \|1_{\{f \le \phi^{1/2}\}}\|_{L^{2/\gamma}(I^{\pm}_{\delta} \to L^{2/\gamma,2}(B_{\delta}))}^2  \\
& \le
K \phi^p \left(\frac{2A_2}{(\delta - \delta')^{k_2}\log \phi} \right)^{\gamma} + K\phi^{p/2},
\end{align*}
for some $K$ depending only on $\gamma$.
The two terms on the right hand side are equal if 
\[ p = \frac{2}{\log \phi} \log \left( \frac{(\delta - \delta')^{k_2} \log \phi}{2A_2} \right)^{\gamma}. \]
We have $p < 1 - \eta$ if $\phi$ is sufficiently large, that is, if 
\begin{align} \label{eq:A_1 2.2.10}
\phi \geq C
\end{align}
for some $C$ depending on $\eta, \gamma, A_2$. 
Hence, for $\phi \geq C$, the first hypothesis of the lemma yields
\begin{align*}
\log \phi(\delta')
& \le \frac{1}{p} \log  (2KA_1(\delta - \delta')^{-k_1})  + \frac{\log \phi}{2} \\
& \le \frac{\log \phi}{2} \left[ \frac{\log (2KA_1(\delta - \delta')^{-k_1})} { \log \left( \frac{(\delta - \delta')^{k_2} \log \phi}{2A_2} \right)^{\gamma}} + 1  \right].
\end{align*}
If 
\begin{align} \label{eq:2.2.12}
\left( \frac{(\delta - \delta')^{k_2} \log \phi}{2A_2} \right)^{\gamma} \ge \left( 2KA_1(\delta - \delta')^{-k_1} \right)^2,
 \end{align}
then 
\[ \log \phi(\delta') \leq \frac{3}{4} \log \phi. \]
On the other hand, if \eqref{eq:2.2.12} or \eqref{eq:A_1 2.2.10} is not satisfied, then
\[ \log \phi(\delta') \leq \log \phi \leq \log C + \frac{2A_2}{(\delta - \delta')^{k_2}} \left( 2KA_1(\delta - \delta')^{-k_1} \right)^{2/\gamma} \le \frac{A}{(\delta - \delta')^{k_2 + 2k_1/\gamma}}, \]
for some $A$ depending on $A_1,A_2,\eta,\gamma$.
In all cases, we have
\begin{align} \label{eq:bombieri 3/4}
 \log \phi(\delta')  
  \le 
  \frac{3}{4} \log \phi(\delta) + \frac{A}{(\delta - \delta')^{k_2 + 2k_1/\gamma}}.
\end{align}
Let $\delta_j = 1 - \frac{1-\delta*}{1+j}$. Iterating \eqref{eq:bombieri 3/4}, we get
\[ \log \phi(\delta^*) 
\le \sum_{j=0}^{\infty} \left( \frac{3}{4} \right)^j \frac{A}{ (\delta_{j+1} - \delta_j)^{k_2 + 2k_1/\gamma}} =: A_3 < \infty. \]
\end{proof}

In order to verify \eqref{eq:log f > lambda} in our context, we need the following ``log lemma" which is based on the weighted Poincar\'e inequality \eqref{eq:weighted PI}. Our proof of the log lemma roughly follows \cite[Lemma 5.4.1]{SC02}.

\begin{lemma}\label{lem:log lemma} 
Let $u$ be a non-negative locally bounded local weak solution of the heat equation for $\e_t$ in $Q^{\pm}$. Suppose H.2 holds for $u$. Suppose \eqref{eq:weighted PI} holds for $f = \log \bar u_{\varepsilon}(t)$ uniformly for all $t$ in $I^{\pm}$, respectively.
Then there exists a constant $c \in (0,\infty)$ depending on $u(a',\cdot)$ or $u(b',\cdot)$, respectively, such that, for all $\lambda >0$, $\delta \in [\delta^*,1)$,
\begin{align*} &\quad \bar\mu( \{ (t,z) \in Q^{\pm}_{\delta} : \pm \log \bar u_{\varepsilon} < - \lambda - (\pm c) \} ) \\
 & \le  \frac{3}{\lambda} \left( 1 \vee \mu(B_1) \right)  |I_{\delta}^{\pm}| \left( \frac{C_{\mbox{\tiny{wPI}}}}{a |I_{\delta}^{\pm}|} 
 +  \sum_{i=1}^m \| D_i \|_{q_i} \left( |I_{\delta}^{\pm}|^{\gamma} \vee |I_{\delta}^{\pm}| \right)  + \frac{C_2 |I_{\delta}^{\pm}|}{|1-\delta|^k} \right).
 \end{align*}
\end{lemma}

\begin{proof}
Let $p=0$ and $\psi = \psi_{\delta,1}$.
 Hence Lemma \ref{lem:steklov supsol} applied with $\chi \equiv 1$ yields
\begin{align*}
\int \log \bar u_{\varepsilon}(t) \psi^2 d\mu 
 -  \int \log \bar u_{\varepsilon}(t-h) \psi^2 d\mu 
 & =  \int \H(u_{\varepsilon}(t)) \psi^2 d\mu 
-  \int \H(u_{\varepsilon}(t-h)) \psi^2 d\mu \\
&  =
- \int_{t-h}^t \e_s \left( u(s),\H'(u_{\varepsilon}(s)) \psi^2 \right) ds,
\end{align*}
for any $t \in I^+_1$ and $h < a'-a$.  Multiplying each side by $\frac{1}{h}$ and letting $h \to 0$,
\begin{align*}
 \frac{d}{dt} \int \log \bar u_{\varepsilon}(t) \psi^2 d\mu 
& =
- \e_t \left( u(t),\H'(u_{\varepsilon}(t)) \psi^2 \right),
\end{align*}
where $\frac{d}{dt}$ denotes taking the left-derivative in $t$.
Thus, by H.2,
\begin{align*}
& \quad \frac{d}{dt} \int \log \bar u_{\varepsilon}(t) \psi^2 d\mu + a \int \bar u_{\varepsilon}(t)^{-2} \psi^2 d\Gamma(u(t)) \\
& \le 
\sum_{i=1}^m D_i(t) \| \psi \|_{2r_i',2}^2 + C_2 |1-\delta|^{-k} \int \psi d\mu =: A_{30}(t)
\end{align*}
for a.e.~$t \in I_1^+$.
Let 
\[ W(t) :=  \frac{\int \log \bar u_{\varepsilon}(t) \psi^2 d\mu }{ \int \psi^2 d\mu}. \]
By \eqref{eq:weighted PI},  
\begin{align*}
 \int |\log \bar u_{\varepsilon}(t) - W(t)|^2 \psi^2 d\mu  \le C_{\mbox{\tiny{wPI}}}  \int \bar u_{\varepsilon}^{-2}(t) \psi^2 d\Gamma(u(t)),
\end{align*}
for a.e.~$t \in I_1^+$.
Hence,
\begin{align*} 
& \quad  \frac{d}{dt} W(t) + \frac{a}{C_{\mbox{\tiny{wPI}}}  \int \psi^2 d\mu}  \int_{B_{\delta}} |\log \bar u_{\varepsilon}(t) - W(t)|^2 \psi^2 d\mu 
 \le
 \frac{A_{30}(t)}{\int \psi^2 d\mu}.
\end{align*}
Writing
\begin{align*}
\overline{w}(t,z) & = \log \bar u_{\varepsilon}(t,z) + \frac{\int_t^{b+a_1} A_{30} ds}{\int \psi^2 d\mu}, \\
\overline{W}(t) & = W(t) +  \frac{\int_t^{b+a_1} A_{30} ds}{\int \psi^2 d\mu},
\end{align*}
we obtain for a.e.~$t \in I_1^+$ that
\begin{align} \label{eq:5.4.2} 
\frac{d}{dt} \overline{W}(t)  + \frac{a}{C_{\mbox{\tiny{wPI}}}  \int \psi^2 d\mu}  \int_{B_{\delta}} |\overline{w} - \overline{W}|^2 \psi^2 d\mu  
\le 0.
\end{align}
Integrating over $(t,b+a_1)$, we find that $\overline{W}(b+a_1) - \overline{W}(t) \le 0$.
For $\lambda > 0$, set
\begin{align*}
\Omega^+_t(\lambda) & = \{ z \in B_{\delta} : \overline{w}(t,z) < - \lambda + \overline{W}(b+a_1) \}.
\end{align*}
Then, for a.e.~$t \in I_1^+$, $z \in \Omega^+_{t}(\lambda)$,
\begin{align} \label{eq:W und lambda} 
\overline{w}(t,z) - \overline{W}(t) < - \lambda + \overline{W}(b+a_1) - \overline{W}(t) \le - \lambda.
\end{align}
Applying \eqref{eq:W und lambda} in the inequality \eqref{eq:5.4.2}, 
\[ \frac{d}{dt} \overline{W}(t) + \frac{a}{C_{\mbox{\tiny{wPI}}}  \int \psi^2 d\mu}  |\lambda  - \overline{W}(b+a_1)+ \overline{W}(t)|^2 \mu(\Omega^+_{t}(\lambda)) 
\le 0.  \]
Dividing by $|\lambda  - \overline{W}(b+a_1)+ \overline{W}(t)|^2$, we can rewrite this inequality as
\begin{align*}
-\frac{d}{dt} |\lambda - \overline{W}(b+a_1) + \overline{W}(t)|^{-1} 
+  \frac{a}{C_{\mbox{\tiny{wPI}}}  \int \psi^2 d\mu} 
 \mu(\Omega^+_{t}(\lambda)) \le 0,
\end{align*}
or, equivalently,
\begin{align} \label{eq:mu(Omega+)}
\begin{split}
  \mu(\Omega^+_t(\lambda)) 
& \le 
\frac{C_{\mbox{\tiny{wPI}}}  \int \psi^2 d\mu}{a} 
 \left( \frac{d}{dt} |\lambda - \overline{W}(b+a_1) + \overline{W}(t)|^{-1} \right).
  \end{split}
\end{align}
Integrating over $I_1^+$,
\begin{align*}
 \overline{\mu}\left( \left\{(t,z) \in Q^+_{\delta} : \log \bar u_{\varepsilon}(t,z) + \frac{\int_t^{b+a_1} A_{30} ds}{\int \psi^2 d\mu} < - \lambda + \overline{W}(b+a_1) \right\} \right)
\le   \frac{C_{\mbox{\tiny{wPI}}}  \int \psi^2 d\mu}{a \lambda}.
\end{align*}

On the other hand,
\begin{align*}
& \quad 
\overline{\mu}\left( \left\{(t,z) \in Q^+_{\delta} : \frac{\int_t^{b+a_1} \sum_{i=1}^m D_i(s) \| \psi \|_{2r_i',2}^2 ds}{\int \psi^2 d\mu} > \frac{\lambda}{3} \right\} \right) \\
& = 
\int_{B_{\delta}} \int_{I^+_{\delta}}  1_{ \big\{\int_t^{b+a_1} \sum_{i=1}^m D_i(s) \| \psi \|_{2r_i',2}^2 ds > \frac{\lambda}{3} \int \psi^2 d\mu \big\}} dt \, d\mu \\
& \le 
\frac{3}{\lambda} \frac{\mu(B_{\delta})}{\int \psi^2 d\mu} \int_{I_{\delta}^+} \int_t^{b+a_1} \sum_{i=1}^m D_i(s) \| \psi \|_{2r_i',2}^2 ds \, dt \\
& \le 
\frac{3}{\lambda} \frac{\mu(B_{\delta})}{\int \psi^2 d\mu} |I_{\delta}^+|  \sum_{i=1}^m \| D_i \|_{q_i}  \left( 1 \vee \mu(B_1) \right) \left( |I_{\delta}^+|^{\gamma} \vee |I_{\delta}^+| \right),
\end{align*}
and
\begin{align*}
& \quad 
\overline{\mu}\left( \left\{(t,z) \in Q^+_{\delta} : \int_t^{b+a_1}\frac{  C_2 |1-\delta|^{-k} \int \psi d\mu }{\int \psi^2 d\mu} dt > \frac{\lambda}{3} \right\} \right) \\
& = 
\overline{\mu}\left( \left\{(t,z) \in Q^+_{\delta} : b+a_1 - a_{\delta} > (b+a_1-t)> \frac{\lambda}{3} \frac{|1-\delta|^k \int \psi^2 d\mu}{C_2  \int \psi d\mu}  \right\} \right) \\
& \le 
|I_{\delta}^+| \left( 1 -  \frac{\lambda}{3} \frac{|1-\delta|^k \int \psi^2 d\mu}{|I_{\delta}^+| C_2  \int \psi d\mu} \right) \mu(B_{\delta}) \\
& \le 
\frac{3}{\lambda}  \frac{|I_{\delta}^+|^2 C_2  \int \psi d\mu}{|1-\delta|^k \int \psi^2 d\mu}  \mu(B_{\delta})
\end{align*}
where we used that $1-x \le \frac{1}{x}$.
The three inequalities above yield
\begin{align*}
& \quad \ \overline{\mu}\left( \big\{(t,z) \in Q^+_{\delta} : \log \bar  u_{\varepsilon}(t,z) < - \lambda + \overline{W}(b+a_1) \big\} \right)  \\
& \le 
 \frac{3}{\lambda} \left( 1 \vee \mu(B_1) \right)  |I_{\delta}^+| \left( \frac{C_{\mbox{\tiny{wPI}}}}{a |I_{\delta}^+|} 
 +  \sum_{i=1}^m \| D_i \|_{q_i} \left( |I_{\delta}^+|^{\gamma} \vee |I_{\delta}^+| \right)  + \frac{C_2 |I_{\delta}^+|}{|1-\delta|^k} \right).
\end{align*}
This proves Lemma \ref{lem:log lemma} when $\pm$ is $+$. When $\pm$ is $-$, the proof follows the same reasoning but uses right-derivatives and the upper Steklov average instead of the lower Steklov average.
\end{proof}

\subsection{Parabolic Harnack inequality} \label{ssec:PHI}

Let $0 < \tau_1 < \tau_2 < \tau_3 < \tau_4 \leq 1$. 
Set
\begin{align*}
 Q^- &= (a + \tau_1, a + \tau_2) \times B_{\delta}, \\ 
 Q^+ &= (a + \tau_3, a + \tau_4) \times B_{\delta}.
\end{align*}
Let $I$ be an open interval containing $[a,a+\tau_4]$ and let $Q = I \times B$.

\begin{theorem}[Parabolic Harnack inequality] \label{thm:PHI}
Let $u$ be a non-negative local weak solution of the heat equation for $\e_t$ in $Q$.
Suppose H.1, H.2 hold for $u$. Suppose the weighted Sobolev inequality \eqref{eq:wSI} holds for $\bar u(t)$ and $\bar u_{\varepsilon}(t)$ uniformly for all $t \in I$ and all small $\varepsilon >0$. Suppose the weighted Poincar\'e inequality \eqref{eq:weighted PI} holds for $\log \bar u_{\varepsilon}(t)$ uniformly for all $t \in I$ and all small $\varepsilon >0$. 
Then there is a constant $C_{\mbox{\em \tiny{PHI}}} \in (0,\infty)$  such that 
\[ \sup_{Q^-} \bar u \leq C_{\mbox{\em \tiny{PHI}}} \inf_{Q^+} \bar u. \]
The constant $C_{\mbox{\em \tiny{PHI}}}$ depends only on $\tau_1$, $\tau_2$, $\tau_3$, $\tau_4$, $\delta$, $\gamma$, $k$, $\nu$, and upper bounds on $C_{\mbox{\em \tiny{wPI}}} \tau_2^{-1}$, $C_{\mbox{\em \tiny{wPI}}} (\tau_4 - \tau_2)^{-1}$, $\frac{C_{\mbox{\em \tiny{SI}}}}{a}$, $C_1 T^{\gamma}$, 
$\left( C_2+2C_3+ \frac{ C_{\mbox{\tiny{SI0}}}}{C_{\mbox{\tiny{SI}}}(\delta',\delta)}\right)T^{\gamma+1}$, $\|D_i\|_{q_i} (T^{\gamma} \vee T)$, 
where $T=\tau_2 \vee (\tau_4 - \tau_2)$.
\end{theorem}

\begin{proof}
By Theorem \ref{thm:u loc bounded}, $u$ is locally bounded, so the mean value estimates of Theorem \ref{thm:MVE supsol} hold. 
Let $A_1=C'A_0^{\frac{\nu+2}{2}}$ and $k_1 = k(\nu+2)$.
By Lemma \ref{lem:log lemma}, there exists a positive constant 
\[ c = \overline{W}(a+\tau_2) = W(a+\tau_2) = \frac{\int \log \bar u_{\varepsilon}(a+\tau_2) \psi^2 d\mu}{\int \psi^2 d\mu} \]
such that the hypotheses of Lemma \ref{lem:Bombieri} are satisfied with $f=(\bar u_{\varepsilon}e^c)$ on $I^+_1 = (a,a+\tau_2)$ and with $f=(\bar u_{\varepsilon}e^c)^{-1}$ on $I_1^- =(a+\tau_2, a+\tau_4)$.
We obtain that there exist positive constants $A_3, A_3'$ such that
\[ \sup_{Q^-} \bar u_{\varepsilon} e^c \leq A_3 \]
and
\[ \sup_{Q^+} (\bar u_{\varepsilon} e^c)^{-1} \leq A_3', \]
for any $\varepsilon \in (0,1)$.

Hence,
\[ \sup_{Q^-} \bar u_{\varepsilon} \leq e^{-c} A_3 \leq A_3 \frac{A_3'}{\sup_{Q^+} \bar u_{\varepsilon}^{-1}} \leq A_3 A_3' \inf_{Q^+} \bar u_{\varepsilon}. \]
Letting $\varepsilon \to 0$ on both sides finishes the proof.
\end{proof}

\section{Examples} \label{sec:examples}

\subsection{Quasilinear forms adapted to a Dirichlet form}

\subsubsection{Dirichlet spaces with induced metric} \label{ssec:Dirichlet spaces}
Let $(X,d)$ be a locally compact separable metric space and $\mu$ a locally finite Borel measure on $X$ with full support.
Any symmetric strongly local regular Dirichlet form $(\e,\F)$ on $L^2(X,\mu)$ induces a pseudo-metric
\begin{align*}
d_{\e}(x,y) := \sup \big\{ f(x)-f(y): f \in \F_{\mbox{\tiny{loc}}} \cap \mathcal{C}(X), \, d\Gamma(f,f) \leq d\mu \big\},
\end{align*}
where $d\Gamma(f,f)$ is the energy measure of $(\e,\F)$, $\mathcal{C}(X)$ is the space of continuous functions on $X$, and 
\[ \F_{\mbox{\tiny{{loc}}}}(U)  :=  \{ f \in L^2_{\mbox{\tiny{loc}}}(U) : \forall \textrm{ compact } K \subset U, \ \exists f^{\sharp} \in \F, f\big|_K = f^{\sharp}\big|_K \mbox{ $\mu$-a.e.} \}. \]
%where $L^2_{\mbox{\tiny{loc}}}(U)$ is the space of locally $L^2$-integrable functions on $U$.
For an open subset $Y \subset X$, we consider
\begin{enumerate}
\item[(A1)] $d_{\e}$ is a (finite, non-degenerate) metric 
which generates the original topology on $X$, 
\item[(A2)] for every $B(x,2R) \subset Y$, the ball $B(x,R)$ is relatively compact.
\end{enumerate}
If (A1) and (A2) are satisfied on $Y$, then there exists a cutoff function for $B(x,R)$ in $B(x,R+r)$ such that
\begin{align} \label{eq:A1 cutoff}
 d\Gamma(\psi,\psi) \le 2 r^{-2} d\mu,
\end{align}
provided that $0 < r \le R$ and $B(x,2R) \subset Y$.

For instance, (A1) and (A2) are satisfied by the canonical Dirichlet forms on $\R^n$, Riemannian manifolds $(M^n,g)$ with Ricci curvature bounded below, or Riemannian complexes (see \cite{PivarskiSC}). These spaces are known to satisfy the volume doubling property and the scale-invariant Poincar\'e inequality up to some scale $R_0 \in (0,\infty]$ which depends on a lower curvature bound. Volume doubling and Poincar\'e inequality imply that for any $x \in X$, $R \in (0,R_0)$, $B_{\delta} = B(x,\delta R)$, and any $f \in \F_{\mbox{\tiny{c}}}(B(x,R))$, the weighted
Sobolev inequality \eqref{eq:wSI} holds with $k=0$ and 
\begin{align*}
C_{\mbox{\tiny{SI}}} &
= C R^2 \mu(B(x,R))^{-2/\nu} , \\
C_{\mbox{\tiny{SI0}}} &
= C' \mu(B(x,R))^{-2/\nu}
\end{align*}
and the weighted Poincar\'e inequality \eqref{eq:weighted PI} holds with 
\begin{align*}
C_{\mbox{\tiny{wPI}}}(\delta',\delta) = C'' R^2.
\end{align*}
for some constant $C,C',C'' \in (0,\infty)$ that may depend on $\delta^*$ but not on $\delta$, $\delta'$.
%In particular, the Sobolev inequality and the weighted Poincar\'e inequality are scale-invariant for $R \in (0,R_0)$ for some $R_0>0$. In the case of a Riemannian manifold $(M^n,g)$, $R_0$ depends only on a lower bound on the Ricci curvature and we may choose $R_0 = \infty$ when $(M^n,g)$ has non-negative curvature.

The parabolic Harnack inequality on Dirichlet spaces satisfying (A1) and (A2) is studied in \cite{SturmIII, LierlPHI} under the hypothesis that the scale-invariant Poincar\'e inequality and the doubling property hold locally on a subset $Y$ up to scale $R_0 >0$, that is, for balls $B(x,R)$ with $B(x,4R) \subset Y$ and $R \le R_0/4$.
Then a scale-invariant parabolic Harnack inequality holds on $Y$ up to scale $R_0$.
Though Sturm does not present the proof of this result in reasonably full detail (cf.~the discussion in \cite{LierlPHI}) and particularly an argument like the chain rules for weak time-derivatives in Section \ref{ssec:steklov chain rule} are not given in \cite{SturmII,SturmIII}, we would like to mention that, in the special case of a symmetric strongly local regular (time-dependent) Dirichlet form as considered in \cite{SturmIII}, it was communicated to the author by K.-T.~Sturm that it is possible to give a simpler proof by replacing $\H_n$ by a twice continuously differentiable function. More precisely, the author has verified that the argument works with
\begin{align*}
\mathcal{H}_n(v) := 
 \frac{1}{2} v^2 (v \wedge n)^{p-2} - \left( 1 - \frac{1}{p-1} \right) v (v \wedge n)^{p-1} +  \left( \frac{1}{2} - \frac{1}{p-1} + \frac{1}{p-1} \frac{1}{p} \right) v_n^p.
\end{align*} 
%Then $\mathcal  H_n$ is twice continuously differentiable,
%\begin{align*}
%\mathcal{H}_n'(v) = 
% v (v \wedge n)^{p-2} - \left( 1 - \frac{1}{p-1} \right) (v \wedge n)^{p-1},
%\end{align*} 
%and  $\mathcal{H}_n''(v) = (v \wedge n)^{p-2}$. 
Unfortunately, it seems that this simpler argument does not extend beyond the special case of symmetric strongly local Dirichlet forms.

\subsubsection{Adapted quasilinear forms satisfy H.1 and H.2} \label{ssec:adapted forms}

In this subsection we show that quasilinear forms that are adapted to a reference Dirichlet form $(\e,\F)$ satisfy hypotheses H.1 and H.2, provided that the underlying space admits appropriate cutoff functions.

\begin{definition} \label{def:adapted form}
We say that a quasilinear form $\e_t$ is {\em adapted to $(\e,D(\e))$} if the domain of $\e_t$ is $\F=D(\e)$, and there is a positive integer $m$ such that
\begin{enumerate}
\item (Generalized uniform coerciveness) 
for all $u \in \F$,
\begin{align} \label{eq:e^A >}
\begin{split}
d\A_t(u,u)  & \ge a d\Gamma(u,u) - \sum_{i=1}^m b_i^2 u^2 d\mu  - \sum_{i=1}^m  w_{1,i}^2 d\mu 
\end{split}
\end{align}
\item (Generalized sector condition) 
for all $u,v \in \F$, $f:X \to \R$ bounded Borel measurable, $g \in L^2(X,d\Gamma(v,v))$,
\begin{align} \label{eq:e^A <} 
\begin{split}
& \quad \left| \int f g \, d\A_t(u,v) \right| \\
& \le 
\left( \bar a \, \left(\int f^2 d\Gamma(u,u) \right)^{1/2}  +  \sum_{i=1}^m \| e_i \|_{r_i,\infty} \| f u \|_{r_i'',2}  +  \sum_{i=1}^m \|w_{3,i}\|_{r_i,\infty} \| f \|_{r_i'',2} \right) \\
& \quad \left( \int g^2 d\Gamma(v,v) \right)^{1/2},
\end{split}
\end{align}
 and, for all $u,v \in \F$, and all bounded Borel measurable functions $f$ and $g$ on $X$, 
\begin{align} \label{eq:e^B <} 
\begin{split}
& \quad \left| \int f g \, d\B_t(u,v) \right| \\
 & \le \sum_{i=1}^m \| c_i \|_{r_i,\infty} \left( \int f^2 d\Gamma(u,u) \right)^{1/2} \| g v \|_{r_i'',2} + \sum_{i=1}^m \|d_i\|_{r_i,\infty} \| f u \|_{2r_i',2}  \| g v \|_{2r_i',2} \\
 & \quad  +  \sum_{i=1}^m \|w_{2,i}\|_{r_i,\infty} \| f \|_{2r_i',2} \| g v \|_{2r_i',2}.
\end{split}
\end{align}
\end{enumerate}
Here, $a$ and $\bar a$ are positive constants and the ``coefficients" $b_i,c_i,d_i,e_i,w_{1,i},w_{2,i},w_{3,i}$ are non-negative functions of $(x,t)$ and each coefficient is in $L^{q_i}(I \to L^{{r_i},\infty}(B))$ for some $(r_i,q_i)$. The pair $(r_i,q_i)$ may be different for each coefficient but, for some fixed $\gamma>0$, all pairs $(r_i,q_i)$ must satisfy \eqref{eq:gamma}.
\end{definition}

\begin{proposition} \label{prop:adapted forms H.1 H.2}
Suppose the reference Dirichlet form $(\e,\F)$ satisfies (A1)-(A2). Let $I$ be a bounded open time-interval and $U \subset X$ open. If $\e_t$ is a quasilinear form adapted to $(\e,\F)$ then $\e_t$ satisfies H.1a, H.1b and H.2 for all $u \in L^2_{\mbox{\tiny{loc}}}(I \to L^2(U))$ with $\kappa = \sum_i \|w_{1,i} \| + \|w_{2,i}\| + \|w_{3,i}\|$.
\end{proposition}

In the following results we assume the volume doubling property and the Poincar\'e inquality ``locally up to scale $R_0>0$". For the precise definitions of these properties, we refer to \cite{LierlPHI}.

\begin{theorem}[Scale-invariant parabolic Harnack inequality] \label{thm:scale-invariant PHI}
Suppose the reference Dirichlet form $(\e,\F)$ satisfies (A1)-(A2), volume doubling and the scale-invariant Poincar\'e inequality on $Y \subset X$ up to scale $R_0>0$. Let $\e_t$ be a quasilinear form adapted to $(\e,\F)$. Then $\e_t$ satisfies the scale-invariant parabolic Harnack inequality up to scale $R_0$:  
There is a positive constant $C_{\mbox{\tiny{PHI}}}$ such that for any $s \in \R$, any ball $B(x,R)$ with $B(x,2R) \subset Y$ and $R \le R_0$,
and for any non-negative local weak solution $u$ for $\e_t$ in $Q = (s,s+\tau R^2) \times B(x,R)$, it holds
\[ \sup_{Q_-} u + \kappa \le C_{\mbox{\tiny{PHI}}} (\inf_{Q^+} u + \kappa), \]
where $Q^- = (s + \frac{1}{4}\tau R^2,s+\frac{1}{2} \tau R^2) \times B(x,\delta R)$ and $Q^+ = (s + \frac{3}{4}\tau R^2,s+ \tau R^2) \times B(x,\delta R)$, and
$\kappa = \sum_i \|w_{1,i} \| + \|w_{2,i}\| + \|w_{3,i}\|$.

The constant $C_{\mbox{\tiny{PHI}}}$ depends only on $\tau$, $\delta$, $a$, $\bar a$, the norms of the coefficients in their respective spaces, the volume doubling constant, the Poincar\'e constant, and - unless $\gamma,b,c,d,e,w_1,w_2,w_3$ all vanish - also on an upper bound on $R_0^2$.
\end{theorem}
%In particular, if $\e_t$ is a bilinear form associated with a second order linear divergence form operator and if $\gamma=0$ then we recover the parabolic Harnack inequality of \cite{LierlPHI} but under sharp integrability conditions on the coefficients.

For the proofs of Proposition \ref{prop:adapted forms H.1 H.2} and Theorem \ref{thm:scale-invariant PHI} we need two lemmas stated below.

\begin{theorem} \label{thm:Hoelder} 
Suppose the reference Dirichlet form $(\e,\F)$ satisfies (A1)-(A2), volume doubling and the scale-invariant Poincar\'e inequality on $Y \subset X$ up to scale $R_0>0$. Let $\e_t$ be a quasilinear form adapted to $(\e,\F)$.

Let $u$ be a non-negative local weak solution of the heat equation for $\e_t$ in $Q = (s,s+\tau R^2) \times B(x,R)$ where $s \in \R$, $B(x,2R) \subset Y$ and $R \le R_0$.
Then $u$ has a continuous version 
which satisfies
 \[ \sup_{(t,y),(t',y')\in Q'} \left\{ \frac{ |u(t,y) - u(t',y')| }{ [ |t-t'|^{1/2} + d(y,y')]^{\alpha} } \right \}
\leq \frac{C}{r^{\alpha} } \sup_{Q} |\bar u| \]
where $Q'= (s+(1-\delta)\tau R^2),s+\tau R^2) \times B(x,\delta R)$.
The constant $C >0$ and the H\"older exponent $\alpha > 0$ depend at most on
$\tau$, $\delta$, $\gamma$, $a$, $\bar a$, the norms of the coefficients $b_i,c_i,d_i,e_i,w_{1,i},w_{2,i},w_{3,i}$ in their respective spaces, the volume doubling constant, the Poincar\'e constant, and - unless the coefficients all vanish - also on an upper bound on $R_0^2$.
\end{theorem}
\begin{proof}
We omit the proof because it is a standard application of the parabolic Harnack inequality which we proved in Theorem \ref{thm:PHI}. See \cite[Theorem 4]{AS67} for details.
\end{proof}

\begin{remark}
Assumptions (A1)-(A2) in Proposition \ref{thm:scale-invariant PHI}, Theorem \ref{thm:Hoelder} and in the maximum principle of Theorem \ref{thm:max principle} can be relaxed: We may instead assume that (A2) holds for metric balls in $(X,d)$, and the cutoff Sobolev inequality on annuli, CSA($\Psi$), holds (see \cite{AB15} for the definition). In this case, the time-space scaling has to be changed in the obvious way from $R^2$ to $\Psi(R)$ in the Poincar\'e inequality, the Sobolev inequality and in Theorem \ref{thm:scale-invariant PHI}, and from $|t-t'|^{1/2}$ to $\Psi^{-1}(|t-t'|)$ in Theorem \ref{thm:Hoelder}. The constants $C$ will then depend also on the constants and exponents appearing in CSA($\Psi$). 
\end{remark}

\begin{lemma} \label{lem:pre-H.1a} If $\e_t$ is a quasilinear form adapted to $(\e,\F)$ with $m=1$ then, for any $t \in \R$, any non-negative $u \in \F$, $\kappa >0$, $n \ge \kappa$ positive integer, $p \in [2,\infty)$, 
\begin{align} \label{eq:adapted form H_n}
\begin{split}
& \quad -  \e_t( u(t), \mathcal H'_n(u(t)) \psi^2 )  + \frac{a}{2} \int \bar u_n^{p-2} \psi^2 d\Gamma(u,u) 
+ (p-2) a \int_{\{\bar u \le n\}} \bar u_n^{p-2} \psi^2 d\Gamma(u,u) \\
& \le 
\left( (p-1)  \left(\|b\|_{r,\infty}^2 + \bigg\|\frac{w_1}{\kappa} \bigg\|_{r,\infty}^2 \right) + \frac{4}{a} \| c \|_{r,\infty}^2 + \| e \|_{r,\infty}^2 + \bigg\|\frac{w_3}{\kappa} \bigg\|_{r,\infty}^2 \right) \left\| \bar u \bar u_n^{\frac{p-2}{2}} \psi \right\|_{r'',2}^2  \\
& \quad + 2 \left( \|d\|_{r,\infty} +  \bigg\|\frac{w_2}{\kappa} \bigg\|_{r,\infty}  \right) \| \bar u_n^{\frac{p-2}{2}} \bar u \psi \|_{2r',2}^2 
+ 4\left(\frac{4\bar a^2}{a} + 1 \right) \int \bar u^2 \bar u_n^{p-2} d\Gamma(\psi,\psi).
\end{split}
\end{align}
\end{lemma}

\begin{proof}
It suffices to give the proof in the case $m=1$.
We use the decomposition
$\e_t(f,g) = \int d\A_t(f,g) +  \int d\B_t(f,g)$ 
and estimate each integral separately.
 We write $u$ for $u(t)$ and $u_n$ for $u_n(t)$. 
By the chain rule, right strong locality and right linearity, we have
\begin{align*}
& \quad - \int d\A_t (u, \mathcal H'_n(u) \psi^2) \\
& = - \int d\A_t(u,(\bar u \bar u_n^{p-2}+\kappa^{p-1})\psi^2) \\
& = - \int \bar u_n^{p-2}\psi^2 d\A_t(u,u)
- \int 2 (\bar u \bar u_n^{p-2} + \kappa^{p-1}) \psi \, d\A_t(u,\psi) 
- (p-2) \int \bar u \bar u_n^{p-3}\psi^2 d\A_t(u,\bar u_n).
\end{align*}
By right strong locality and right linearity,
\begin{align*}
 \int \bar u \bar u_n^{p-3}\psi^2 d\A_t(u,\bar u_n) 
& =  \int_{\{\bar u \le n\}} \bar u \bar u_n^{p-3}\psi^2 d\A_t(u,\bar u_n) 
 =  \int_{\{\bar u \le n\}} \bar u_n^{p-2} \psi^2 d\A_t(u,u).
\end{align*}
Thus, by \eqref{eq:e^A >} and \eqref{eq:e^A <},
\begin{align*}
& \quad - \int d\A_t (u, \mathcal H'_n(u(t)) \psi^2) \\
& = - \int \bar u_n^{p-2}\psi^2 d\A_t(u,u)
- \int 2 (\bar u \bar u_n^{p-2}+\kappa^{p-1})\psi \, d\A_t(u,\psi) 
- (p-2) \int_{\{ \bar u \le n \}} \bar u_n^{p-2} \psi^2 d\A_t(u,u) \\
& \le 
-  a  \int \bar u_n^{p-2} \psi^2 d\Gamma(u,u) 
-  (p-2) a  \int_{\{\bar u \le n\}} \bar u_n^{p-2} \psi^2 d\Gamma(u,u) \\
& \quad + (p-1) \left[ \int b^2 u^2 \bar u_n^{p-2} \psi^2 d\mu  + \int w_1^2 \bar u_n^{p-2} \psi^2 d\mu \right] \\
& \quad + 2 \left( \bar a \left( \int \bar u_n^{p-2} \psi^2 \Gamma(u,u) \right)^{1/2} + \| e \|_{r,\infty} \| u \bar u_n^{\frac{p-2}{2}} \psi\|_{r'',2} + \|w_3\|_{r,\infty} \|\bar u_n^{\frac{p-2}{2}} \psi \|_{r'',2} \right)\\
& \quad  \left( \int \bar u^2 \bar u_n^{p-2} d\Gamma(\psi,\psi) \right)^{1/2} \\
& \quad + 2 \left( \bar a \left( \int \kappa^{p-2} \psi^2 \Gamma(u,u) \right)^{1/2} + \| e \|_{r,\infty} \| u\kappa^{\frac{p-2}{2}} \psi\|_{r'',2} + \|w_3\|_{r,\infty} \|\kappa^{\frac{p-2}{2}} \psi \|_{r'',2} \right) \\
& \quad \left( \int \bar \kappa^p d\Gamma(\psi,\psi) \right)^{1/2}.
\end{align*}
By right linearity, the chain rule, and \eqref{eq:e^B <},
\begin{align*}
& \quad - \int d\B_t(u,\H_n'(u)\psi^2) \\
& = - \int d\B_t(u,(\bar u \bar u_n^{p-2} + \kappa^{p-1}) \psi^2) \\
& = - \int \bar u_n^{p-2} \psi \, d\B_t(u,\bar u \psi)
- \int \kappa^{p-2} \psi \, d\B_t(u,\kappa \psi) \\
& \le 
 \| c \|_{r,\infty} \left( \int \bar u_n^{p-2} \psi^2 d\Gamma(u,u) \right)^{1/2} \|  \bar u_n^{\frac{p-2}{2}} \bar u \psi \|_{r'',2} \\
& \quad 
+ \left(  \|d\|_{r,\infty} \| u \bar u_n^{\frac{p-2}{2}} \psi \|_{2r',2}  +  \|w_2\|_{r,\infty} \| \bar u_n^{\frac{p-2}{2}} \psi \|_{2r',2} \right) \| \bar u_n^{\frac{p-2}{2}} \bar u \psi \|_{2r',2} \\
& \quad 
+ \| c \|_{r,\infty} \left( \int \kappa^{p-2} \psi^2 d\Gamma(u,u) \right)^{1/2} \|  \kappa^{\frac{p}{2}}  \psi \|_{r'',2} \\
& \quad 
+ \left(  \|d\|_{r,\infty} \| u \kappa^{\frac{p-2}{2}} \psi \|_{2r',2}  +  \|w_2\|_{r,\infty} \| \kappa^{\frac{p-2}{2}} \psi \|_{2r',2} \right) \| \kappa^{\frac{p}{2}} \psi \|_{2r',2}.
\end{align*}
Combining the above estimates and using the fact that
$\kappa \le \bar u_n$ for $n \ge \kappa$,
\begin{align*}
& \quad -  \e_t( u(t), \mathcal H'_n(u(t)) \psi^2 )  + a \int \bar u_n^{p-2} \psi^2 d\Gamma(u,u) 
+ (p-2) a \int_{\{\bar u \le n\}} \bar u_n^{p-2} \psi^2 d\Gamma(u,u) \\
& \le 
(p-1)  \int \left(b^2 + \frac{w_1^2}{\kappa^2}\right) \bar u^2 \bar u_n^{p-2} \psi^2 d\mu  \\
& \quad + 4 \left( \bar a \left( \int \bar u_n^{p-2} \psi^2 \Gamma(u,u) \right)^{1/2} + \| e \|_{r,\infty} \| u \bar u_n^{\frac{p-2}{2}} \psi\|_{r'',2} + \|w_3\|_{r,\infty} \|\bar u_n^{\frac{p-2}{2}} \psi \|_{r'',2} \right) \\
& \quad \left( \int \bar u^2 \bar u_n^{p-2} d\Gamma(\psi,\psi) \right)^{1/2} \\
& \quad
+ 2\| c \|_{r,\infty} \left( \int \bar u_n^{p-2} \psi^2 d\Gamma(u,u) \right)^{1/2} \|  \bar u_n^{\frac{p-2}{2}} \bar u \psi \|_{r'',2} \\
& \quad 
+ 2\left(  \|d\|_{r,\infty} \| u \bar u_n^{\frac{p-2}{2}} \psi \|_{2r',2}  +  \|w_2\|_{r,\infty} \| \bar u_n^{\frac{p-2}{2}} \psi \|_{2r',2} \right) \| \bar u_n^{\frac{p-2}{2}} \bar u \psi \|_{2r',2}.
\end{align*}
By \eqref{eq:lorentz power} and \eqref{eq:lorentz-hoelder},
\begin{align*}
\| b^2 \bar u^2 \bar u_n^{p-2} \psi^2 \|_{1,1}
 \le 
\|b \bar u \bar u_n^{\frac{p-2}{2}} \psi\|_{2,2}^2 
 \le 
\| b \|_{r,\infty}^2 \|\bar u \bar u_n^{\frac{p-2}{2}} \psi\|_{r'',2}^2,
\end{align*}
and similarly for $\frac{w_1}{\kappa}$ in place of $b$.
Now the assertion follows from Young's inequality.
\end{proof}

\begin{lemma} \label{lem:pre-H.1b}
If $\e_t$ is a quasilinear form adapted to $(\e,\F)$ with $m=1$ then, for any $t \in \R$,
\begin{align*}
\begin{split}
& \quad 
\frac{1-p}{|1-p|} \e_t( u(t), \mathcal H'_{\varepsilon}(u(t)) \psi^2 )  
+ |p-1| \frac{a}{4} \int \bar u_{\varepsilon}^{p-2} \psi^2 d\Gamma(u,u) \\
& \le 
 \left( |p-1| \left(\|b\|_{r,\infty}^2 + \bigg\|\frac{w_1}{\kappa} \bigg\|_{r,\infty}^2 \right) + \| e \|_{r,\infty}^2 + \bigg\|\frac{w_3}{\kappa} \bigg\|_{r,\infty}^2 + \frac{1}{|p-1|a} \|c\|_{r,\infty}^2  \right) \|\bar u_{\varepsilon}^{p/2} \psi\|_{r'',2}^2 \\
& \quad 
+ \left( \| d \|_{r,\infty} + \bigg\|\frac{w_2}{\kappa}\bigg\|_{r,\infty} \right) \|\bar u_{\varepsilon}^{\frac{p}{2}} \psi \|_{2r',2}^2 
+ \left( \frac{4\bar a^2}{a |p-1|} + 1 \right) \int  \bar u_{\varepsilon}^p  d\Gamma(\psi,\psi).
\end{split}
\end{align*}
for all non-negative locally bounded $u \in \F$, $t \in \R$, $p \in (-\infty,1-\eta) \cup (1+\eta,2)$.
\end{lemma}

\begin{proof}
By the chain rule, right strong locality and right linearity, we have
\begin{align*}
 \int d\A_t (u, \mathcal H'_{\varepsilon}(u) \psi^2) 
& = (p-1) \int \bar u_{\varepsilon}^{p-2}\psi^2 d\A_t(u,u)
+ \int 2 \bar u_{\varepsilon}^{p-1}  \psi \, d\A_t(u,\psi).
\end{align*}
Thus, by \eqref{eq:e^A >} and \eqref{eq:e^A <},
\begin{align*}
& \quad \frac{1-p}{|1-p|} \int d\A_t (u, \mathcal H'_{\varepsilon}(u) \psi^2) \\
& \le 
|p-1| \left[-  a  \int \bar u_{\varepsilon}^{p-2} \psi^2 d\Gamma(u,u)  
 +  \int b^2 u^2 \bar u_{\varepsilon}^{p-2} \psi^2 d\mu  + \int w_1^2 \bar u_{\varepsilon}^{p-2} \psi^2 d\mu \right] \\
& \quad 
+ 2 \left( \bar a \left( \int \bar u_{\varepsilon}^{p-2} \psi^2 \Gamma(u,u) \right)^{1/2} + \| e \|_{r,\infty} \|u \bar u_{\varepsilon}^{\frac{p-2}{2}} \psi\|_{r'',2} + \|w_3\|_{r,\infty} \| \bar u_{\varepsilon}^{\frac{p-2}{2}} \psi \|_{r'',2} \right) \\
& \quad \left( \int  \bar u_{\varepsilon}^p  d\Gamma(\psi,\psi) \right)^{1/2}.
\end{align*}
By the chain rule and \eqref{eq:e^B <},
\begin{align*}
 \left|\int d\B_t(u,\H_{\varepsilon}'(u)\psi^2) \right| 
& = \left| \int \bar u_{\varepsilon}^{p-1} \psi \, d\B_t(u,\psi) \right| \\
& \le 
\|c\|_{r,\infty} \left( \int \bar u_{\varepsilon}^{p-2} \psi^2 d\Gamma(u,u) \right)^{1/2} \|\bar u_{\varepsilon}^{\frac{p}{2}} \psi\|_{r'',2} \\
& \quad 
+ \left( \| d \|_{r,\infty} \|u \bar u_{\varepsilon}^{\frac{p-2}{2}} \psi \|_{2r',2} + \|w_2\|_{r,\infty} \| \bar u_{\varepsilon}^{\frac{p-2}{2}} \psi\|_{2r',2} \right)
\| \bar u_{\varepsilon}^{\frac{p}{2}} \psi \|_{2r',2}.
\end{align*}
Now the assertion follows from the fact that $\kappa \le \bar u$, \eqref{eq:lorentz power}, \eqref{eq:lorentz-hoelder}, and Young's inequality.
\end{proof}

\begin{remark} \label{rem:pre-H}
It is clear that Lemma \ref{lem:pre-H.1a} and Lemma \ref{lem:pre-H.1b} generalize in the obvious way to the case $m>1$.
\end{remark}

\begin{proof}[Proof of Proposition \ref{prop:adapted forms H.1 H.2}]
This is immediate from Lemma \ref{lem:pre-H.1a}, Lemma \ref{lem:pre-H.1b}, Remark \ref{rem:pre-H}, H\"older's inequality and \eqref{eq:A1 cutoff}.
\end{proof}
%By H\"older's inequality,
%\begin{align*}
%& \quad \int_{I_{\delta'}^-} \left( p^2 \sum C_1^2 ||\bar u^{\frac{p}{2}} \psi||_{p'',2}^2 
%+ p\sum C_2  ||\bar u^{\frac{p}{2}} \psi||_{2p',2}^2  \right) dt \\
%& \le
%\left( p^2 \sum \|C_1\|_q^2 \, 
%+ p \sum \|C_2\|_q \right) |||\bar u^{\frac{p}{2}} \psi|||_{I_{\delta}^- \times B_{\delta}}^2.
%\end{align*}}

\begin{proof}[Proof of Theorem \ref{thm:scale-invariant PHI}]
Apply Proposition \ref{prop:adapted forms H.1 H.2}, Lemma \ref{lem:pre-H.1a}, Lemma \ref{lem:pre-H.1b}, Remark \ref{rem:pre-H}, and Theorem \ref{thm:PHI} with $\kappa = 0$.
\end{proof}

\begin{theorem}[Maximum Principle] \label{thm:max principle}
Let $\e_t$ be a quasilinear form adapted to $(\e,\F)$. Suppose $(\e,\F)$ satisfies (A1)-(A2), volume doubling and Poincar\'e inequality.
Let $u$ be a local weak solution of the heat equation for $\e_t$ in $Q=(s,T) \times U$ where $U \subset X$ is an open subset. Let $M \in \R$ and suppose $(u+M)^+(t) \in \F^0(U)$ for every $t \in (s,T)$ and $(u+M)^+(t) \to 0$ in $L^2(U)$ as $t \to 0$. 
 Then
\begin{align*}
u(t,x) \le  M + C \big( (\|b\| + \|d\|)|M| +  \kappa \big) \quad \mbox{ a.e.~in }Q,
\end{align*}
where $\kappa =  \|w_1\| + \|w_2\|$ and the constant $C$ depends only on $(T-s)$, $\mu(U)$, $\gamma$, $\nu$, $C_{\mbox{\tiny{SI}}}$, and the norms of the coefficients in their respective spaces.
\end{theorem}

\begin{proof}
We first prove the maximum principle in the case $M=0$.
Let $(t,x) \in Q$. Choose an appropriate increasing sequence of neighborhoods $(x,t) \in Q_{\delta'}^- \subsetneq Q_{\delta}^- \subset Q$  satisfying \eqref{eq:C_3}, for all $\frac{1}{2} \le \delta' < \delta \le 1$.
Applying the mean value estimate of Theorem \ref{thm:MVE subsol p>2} and Lemma \ref{lem:|||f|||},
\begin{align*}
\bar u^2 (t,x)
& \le C'(\nu) A_0^{\frac{\nu + 2}{2}} |||\bar u|||_{Q_{\delta'}^-}^2 \\
& \le 2 C'(\nu) A_0^{\frac{\nu + 2}{2}} |I_{\delta'}^-|^{\gamma}  \left( C_{\mbox{\tiny{SI}}}  \int_{I_{\delta}^-} \int  \psi^2 d\Gamma(u) dt +  \sup_{t \in I_{\delta}^-} \| \bar u \|^2_{L^2(B_{\delta})} \right) \\
\end{align*}
where
\[ A_0 :=  \frac{ 32 |I_{\delta}^-|^{\gamma} C_{\mbox{\tiny{SI}}} }{a} (C_1 + 1_{\{\kappa>0\}}) . \]
To estimate the right hand side, we repeat the reasoning in the proof of Lemma \ref{lem:sup_t estimate}, except that we can omit $\psi$ and $\chi$ due to the boundary condition and therefore $K=(T-s) \mu(B_1) \kappa^2$.
\begin{align*}
\sup_{Q_{\delta''}^-} \bar u^2 
& \le C (T-s) \mu(B_1) \kappa^2,
\end{align*}
where the constant $C$ depends only on $(T-s)$, $\mu(U)$, $C_{\mbox{\tiny{SI}}}$, $\nu$, $\gamma$, and the norms of the coefficients in their respective spaces. This completes the proof in the case $M=0$.

If $M \ne 0$, notice that $u-M$ satisfies the zero boundary conditions, and $u-M$ is a local weak subsolution to the heat equation for the quasilinear form
\[ \e_t^M(f,g):=\e_t(f+M,g). \]
Since $(\e,\F)$ is also adapted to $(\e,\F)$ we can now apply the case $M=0$ to $u-M$. Just note that $\kappa$ must be replaced by $\kappa^M = (\|b\| + \|d\|)|M| + \|w_1\| + \|w_2\|$, see \cite[Proof of Theorem 1]{AS67}. 
\end{proof}

Further standard applications of the parabolic Harnack inequality apply to the present setting, for instance, the elliptic Harnack inequality, and various pointwise estimates for weak solutions. Since these applications are well-known and to avoid repetition we keep this section short and only state the following pointwise estimate. For further results see, e.g., \cite[Theorem 5']{AS67} and \cite[Section 5.4.3]{SC02}. 

\begin{theorem}[Pointwise estimate]
Let $\e_t$ be a quasilinear form adapted to $(\e,\F)$. Suppose $(\e,\F)$ satisfies (A1)-(A2), volume doubling and Poincar\'e inequality up to scale $R>0$.
Then there is a constant $C \in (0,\infty)$ such that the following pointwise inequality holds. Suppose there is a continuous curve of length $d$ joining two points $x,y \in X$. Let $U$ be a $\delta$-neighborhood of this curve where $\delta > 0$.
Let $0 < s < t < T$ and let $u$ be a non-negative local weak solution of the heat equation for $\e_t$ in $Q=(0,T) \times U$. Then
\begin{align*}
\log \frac{u(s,x) + \kappa}{ u(t,y) + \kappa } \le C \left( 1 + \frac{t-s}{R^2} + \frac{t-s}{s} + \frac{t-s}{\delta^2} + \frac{d^2}{t-s} \right), 
\end{align*}
where $\kappa = \sum_i \|w_{1,i} \| + \|w_{2,i}\| + \|w_{3,i}\|$.
\end{theorem}

\begin{proof}
This follows by applying the parabolic Harnack inequality of Theorem \ref{thm:scale-invariant PHI} successively along a Harnack chain connecting $x$ to $y$ within $U$. For details, we refer to \cite[Proof of Corollary 5.4.4]{SC02}. 
\end{proof}

\subsection{The structural hypotheses of Aronson-Serrin}
Let $(M^n,g)$ be a smooth complete Riemannian manifold without boundary with Riemannian volume element $d\mu$. 
Let $\e(u,g) = \int_{M^n} \nabla u \, \nabla g \, d\mu$ for $u,g \in \F = W^{1,2}(M^n)$.
Suppose that $M^n$ has a lower Ricci curvature bound. Then the volume doubling property and the Poincar\'e inequality are known to hold locally. It is also clear that suitable cutoff functions exist in the present setting.  In particular, the weighted Sobolev inequality \eqref{eq:wSI} and the weighted Poincar\'e inequality \eqref{eq:weighted PI} hold locally. 

We define
 $$\e_t(u,g) : = \int_{M^n}  \A(x,t,u,\nabla u) \nabla g \, d\mu(x) + \int_{M^n} \mathcal B(x,t,u,\nabla u) g \, d\mu(x), $$ where $\mathcal A(x,t,u,\vec p)$ is a vector function, $\mathcal B(x,t,u,\vec p)$ is a scalar function, defined and measurable for all $t \in \R$, $x \in M^n$, and all values of $u$ and $\vec p$. We require $\A$ and $\B$ to satisfy the structural inequalities \cite[(2)]{AS67}, that is,
\begin{align*}
\vec p \cdot \mathcal A(x,t,u,\vec p) & \ge a |\vec p|^2 - b^2 u^2 - w_1^2 \\
|\mathcal B(x,t,u,\vec p)| & \le c |\vec p| + d |u| + w_2 \\
|\mathcal A(x,t,u,\vec p)| & \le \bar a |\vec p| + e |u| + w_3,
\end{align*}
where $a$ and $\bar a$ are positive constants and $b,c,d,e,w_1,w_2,w_3$ are non-negative functions of $(x,t)$ each contained in an $L^q(I \to L^{r,\infty}(M^n))$ space, where the pair $(r,q)$ may be different for each coefficient but must satisfy
\begin{align*}
& r>2 \mbox{  and  } \frac{n}{2r} + \frac{1}{q} < \frac{1}{2} \quad \mbox{ for } b,c,e,w_1,w_3, \\
& r>1 \mbox{  and  } \frac{n}{2r} + \frac{1}{q} < 1 \quad \mbox{ for } d,w_2.
\end{align*}

Then $\e_t$ is adapted to the Dirichlet form generated by the Laplace-Beltrami operator on $M^n$. Therefore, the scale-invariant parabolic Harnack inequality of Theorem \ref{thm:scale-invariant PHI}, as well as all results of Section \ref{sec:MVE} hold.

In the special case $M^n = \R^n$, we recover the parabolic Harnack inequality of \cite[Theorem 3]{AS67} but under weaker conditions on the coefficients: Indeed, the original conditions \cite[(2)]{AS67} involved $L^r$ in place of the Lorentz space $L^{r,\infty} \subset L^r$.

\subsection{Bilinear forms}
In this subsection, we relate the notion of quasilinear forms to the bilinear forms considered in \cite{LierlPHIf}.

Let $\e_t$ be a bilinear form satisfying Assumption 0 in \cite{LierlPHIf} with respect to a reference form $(\e,\F)$. Suppose the reference form satisfies (A1) and (A2) of Section \ref{ssec:Dirichlet spaces}. Formally, write
\begin{align*}
\int f \, d\A_t(u,g) 
& = \int f d\Gamma_t(u,g) + \r_t(fu,g), \\
\int f \, d\B_t(u,g) 
& = \l_t(fu,g) + \e^{\mbox{\tiny{sym}}}_t(fug,1).
\end{align*}
If $\A_t$ and $\B_t$ are signed measures and if $|\e^{\mbox{\tiny{sym}}}(fg,1)| \le C_* \|f\|_2 \|g\|_{\F}$ for all $f,g \in \F \cap \mathcal{C}_{\mbox{\tiny{c}}}$
then $\e_t$ is indeed a quasilinear form in the sense of Definition \ref{def:quasilinear form}. If in addition $\e_t$ satisfies Assumption 1 and Assumption 2 of \cite{LierlPHI} uniformly in $t$, then our structural hypotheses H.1 and H.2 are satisfied. This is remarkable because it seems that Assumptions 0, 1, 2 do not imply that $\e_t$ would be adapted to $(\e,\F)$ in the sense of Definition \ref{def:adapted form}.

\subsection{Doob's transform}

Consider a non-symmetric divergence form operator on $\R^n$,
\[ L = \sum_{i,j=1}^n \partial_i ( a_{ij} \partial_j), \]
with bounded measurable coefficients $a_{ij}$. Assume that its symmetric part is uniformly elliptic, that is, there exists a constant $c >0$ such that
\[ c |\xi|^2 \le \sum a_{i,j} \xi_i \xi_j, \quad \forall \xi,\zeta \in \R^n. \]
It is clear that the bilinear form associated with $L$ satisfies H.1 and H.2. We also have the Poincar\'e inequality and the localized Sobolev inequality.

Let $U$ be an unbounded inner uniform domain in $\R^n$ with harmonic profile $h>0$ for the Dirichlet Laplacian on $U$. By \cite{GyryaSC}, the Doob's transform $\e^{D,h^2}_U(f,f) = \sum_{i=1}^n \int_U |\partial_i f|^2 h^2 dx$ with domain $\F^h(U) = \frac{1}{h} \F(U)$ satisfies volume doubling and the Poincar\'e inequality. 

Let 
\[ \e^h(f,g) = \sum_{i,j=1}^n \int_U a_{ij} \partial_i (h f) \partial_j (h g) dx. \]

\begin{proposition}
The $h$-transformed bilinear form $\e_h$ is adapted to the reference Dirichlet form $(\e^{D,h^2}_U,\F^h(U))$.
\end{proposition}

Similar results hold for bounded inner uniform domains and for locally inner uniform domains in Euclidean space, and more generally in Harnack-type Dirichlet spaces. The proof will be presented in a forthcoming paper by the author, along with new and sharp two-sided estimates for the Dirichlet heat kernel on $U$ associated with $L$.

%{Degenerate operators arising in population biology}

\subsection{Kolmogorov-Fokker-Planck operator} \label{ssec:Kolmogorov}

Consider the operator 
\[ L u =  \sum_{i,j=1}^{m} \partial_{x_i} (a_{ij} \partial _{x_j} u) + \langle Bx,\nabla u\rangle, \]
where $m \le n$, the coefficients $a_{ij}$ are real-valued measurable functions of $(t,x) \in \R \times \R^n$ satisfying $a_{ij} = a_{ji}$ and 
\[ c |\xi|^2 \le \sum a_{ij} \xi_i \xi_j \le C |\xi|^2, \qquad \forall \xi \in \R^m, \]
and $B$ is a constant $n \times n$ real matrix such that there is a basis of $\R^n$ in which $B$ takes the form
\[ B = \begin{pmatrix}
0 & 0 & \ldots & 0 & 0 \\
B_1 & 0 & \ldots & 0 & 0 \\
0 & B_2 & \ldots & 0 & 0 \\
\vdots & \vdots & \ddots & \vdots & \vdots \\
0 & 0 & \ldots & B_{\tilde k} & 0
\end{pmatrix}, \]
where $B_j$ is an $m_{j} \times m_{j+1}$ matrix of rank $m_{j+1}$, $j = 1,2,\ldots,\tilde k$ with 
\[ m =: m_1 \ge m_2 \ge \ldots \ge m_{\tilde k + 1} \ge 1  \quad \mbox{ and }
m_1 + m_2 + \ldots + m_{\tilde k + 1} = n. \]
Then $L$ is associated with a quasilinear form $\e_t$ which satisfies H.1 with $\Gamma(u,u) = \sum_{i=1}^m |\partial_{x_i} u|^2$. Indeed, integrating by parts we can treat $<Bx,\nabla u>$ like a zero order term. However, H.2 is apparently not satisfied, indicating that H.2 has a structural content that is not already captured by H.1.

The Kolmogorov-Fokker-Planck operator $L$ is an example of a class of subelliptic operators to which the Moser iteration applies, see \cite[Example 1.2]{CiPo08} and \cite{CiPaPo08}.
By \cite[Theorem 3.3]{CiPo08}, a localized Sobolev inequality holds for weak solutions\footnote{We remark that the notion of weak solutions in \cite{CiPo08} is slightly stronger than considered here since \cite{CiPo08} requires the existence of a weak time-derivative that is locally in $L^2$.} to the heat equation associated with $L$ in $Q = (-1,1) \times B(x,1)$, for any $x \in \R^n$, $B_{\delta} = B(x,\delta)$. The localized Sobolev inequality implies the weighted Sobolev inequality \eqref{eq:wSI} of Definition \ref{def:wSI} with $k=2$.

A weighted Poincar\'e inequality for $L$ is not known. This is possibly related to the failure of H.2.

Nevertheless, H.1 and the Sobolev inequality are sufficient to obtain the mean value estimates of Theorem \ref{thm:MVE subsol p>2} and Theorem \ref{thm:MVE supsol}. For the operator $L$ given above, these mean value estimates are already known from \cite[Theorem 1.2 and Corollary 1.4]{PaPo04} and \cite[Theorem 1.4]{CiPo08}. However,  Theorem \ref{thm:MVE subsol p>2} and Theorem \ref{thm:MVE supsol} also apply to Kolmogorov-type operators on more general spaces, such as Euclidean complexes or Riemannian manifolds. For instance, if $(M^n,g)$ is a smooth Riemannian manifold then we can define a Kolmogorov-type operator on $M^n \times M^n$ as
\[ Lu = L_{\mathcal V} u - v \nabla_{\mathcal H} u,
\] 
 where $L_{\mathcal V}$ is a vertical uniformly elliptic diffusion operator and $\nabla_{\mathcal H}$ is the horizontal gradient.

%\bibliographystyle{siam}
%\bibliography{bibliography}

\def\cprime{$'$} \def\cprime{$'$}

Janna Lierl, University of Connecticut, 341 Mansfield Road, Storrs, CT 06250.
janna.lierl@uconn.edu

\end{document}